\documentclass[pdflatex,sn-mathphys-num]{sn-jnl}% Math and Physical Sciences Numbered Reference Style
\usepackage{graphicx}%
\usepackage{multirow}%
\usepackage{amsmath,amssymb,amsfonts}%
\usepackage{amsthm}%
\usepackage{mathrsfs}%
\usepackage[title]{appendix}%
\usepackage{xcolor}%
\usepackage{textcomp}%
\usepackage{manyfoot}%
\usepackage{booktabs}%
\usepackage{float}% Provides the [H] figure-placement option.
\usepackage{placeins}% Keeps the scientific figures within their example subsections.
\usepackage{listings}%
\usepackage{subcaption}%
\usepackage[shortlabels]{enumitem}
\usepackage[numbers,sort&compress]{natbib}% Citation support using natbib.sty
\bibpunct[, ]{[}{]}{,}{n}{,}{,}% Citation support using natbib.sty
\renewcommand\theenumi{(\roman{enumi})}

\newif\ifrevisionmarks
\ifdefined\ShowRevisions\revisionmarkstrue\else\revisionmarksfalse\fi

\newcommand{\esssup}{\operatorname*{ess\,sup}}
\newcommand{\loc}{\mathrm{loc}}

\newcommand{\one}{\mathbf{1}}

\newcommand{\R}{\mathbb{R}}

\newcommand{\B}{\mathbb{B}}

\newcommand{\dH}{\mathbf{d}_H}
\newcommand{\dist}{\mathbf{d}}

\theoremstyle{thmstyleone}%
\newtheorem{theorem}{Theorem}[section]% meant for sectionwise numbers
\newtheorem{proposition}[theorem]{Proposition}% 
\newtheorem{lemma}[theorem]{Lemma}
\newtheorem{corollary}[theorem]{Corollary}

\theoremstyle{thmstyletwo}%
\newtheorem{remark}[theorem]{Remark}%

\theoremstyle{thmstylethree}%
\newtheorem{definition}[theorem]{Definition}%

\usepackage{empheq}
\definecolor{myblue}{rgb}{.8, .8, 1}

\usepackage[numbers,sort&compress]{natbib}% Citation support using natbib.sty

\bibpunct[, ]{[}{]}{,}{n}{,}{,}% Citation support using natbib.sty

\begin{document}

\title[Equilibria in Motion]{Equilibria in Motion:
Stability, Tracking, and Convergence}

%%=============================================================%%
%% GivenName	-> \fnm{Joergen W.}
%% Particle	-> \spfx{van der} -> surname prefix
%% FamilyName	-> \sur{Ploeg}
%% Suffix	-> \sfx{IV}
%% \author*[1,2]{\fnm{Joergen W.} \spfx{van der} \sur{Ploeg} 
%%  \sfx{IV}}\email{iauthor@gmail.com}
%%=============================================================%%

\author{\fnm{Hassan} \sur{Saoud}}\email{saoud.h@gust.edu.kw}

\affil{
\orgdiv{Department of Mathematics and Natural Sciences},
\orgname{Gulf University for Science and Technology},
\orgaddress{\city{Hawally}, \postcode{32093},
\country{Kuwait}}
}

%%==================================%%
%% Sample for unstructured abstract %%
%%==================================%%

\abstract{We study nonautonomous systems whose equilibrium sets are
nonisolated and move with time. Their motion is measured on each ball by
a localized Attouch--Wets distance, which yields an equilibrium speed
that may depend on the radius. The distance from a trajectory to the
current equilibrium set is estimated in two ways: through a composite
Lyapunov function compared with a power of this distance, and, without
any Lyapunov function, through a decay condition on the frozen vector
fields. In the first setting, we prove practical stability, and Lyapunov
stability when the dissipation has no residual, a case that forces the
equilibrium sets to be nested. In the second, we obtain, in the region
where the decay condition holds, an explicit tracking estimate driven by
the equilibrium speed and the inputs, and we prove asymptotic tracking
whenever their averages over unit time intervals vanish, a condition
that cannot be weakened among conditions on the speed alone. We also
derive input-to-state stability bounds and explicit convergence rates.
When the equilibrium speeds are integrable, the equilibrium sets
converge in the Attouch--Wets sense to a limiting set at an explicit
rate. If, moreover, the velocity is bounded by a multiple of the
distance to the current equilibrium set, every trajectory remaining in
the region considered has finite length and converges to a point of the
limiting set; for a fixed equilibrium set, this yields semistability.
The results are illustrated on a scalar example and on
resource-allocation models with time-varying demand and with rotating
weights.}
%%================================%%
%% Sample for structured abstract %%
%%================================%%
\keywords{Nonautonomous systems;
Moving equilibrium sets;
Nonisolated equilibria;
Lyapunov stability;
Tracking;
Input-to-state stability;
Attouch--Wets convergence;
Resource allocation.}

%%\pacs[JEL Classification]{D8, H51}

\pacs[MSC Classification]{34D20, 34D05, 93D05, 93D09, 34C11}

\maketitle
%%%%%%%%%%%%%%%%%%%%%%%%%%%%%
\section{Introduction}
\label{sec:Introduction}

Lyapunov methods are a central tool in the analysis of the asymptotic
behavior of nonlinear dynamical systems, and invariance principles,
stability of sets and robustness notions have considerably extended
their range of applicability \cite{khalil,BacciottiRosier,Haddad}. In
many problems, however, the equilibria are not isolated. They form a
continuum, and the relevant questions are whether every trajectory
converges to one of them and whether the limit is stable. This led to
the notion of semistability, whose Lyapunov analysis was developed by
Bhat and Bernstein \cite{bhat99,bhat03}, and to that of pointwise
asymptotic stability \cite{goebel10,goebel18}. These notions have since
been studied for differential equations, differential inclusions and
hybrid systems \cite{bhat10,saoud15,SaoudTheraDao2026}. In
\cite{bhat10}, convergence is related to the arc length of the orbits:
a system whose orbits all have finite arc length is convergent, and
Lyapunov tests for convergence and semistability are derived from an
inequality between the vector field and the derivative of a function
that need not be positive definite. A related line of research, based
on Matrosov functions and on extensions of LaSalle's invariance
principle, gives stability and convergence toward equilibria and
invariant sets under weak dissipation and explicit time dependence
\cite{Matrosov1962,Loria2005,Mazenc2007,AstolfiPraly2011}. In most of
these works, the set relative to which convergence is studied is fixed.

In many applications, however, the equilibria move. Resource-allocation
and economic adjustment models are typical examples
\cite{Arrow1958,Cherukuri2016}: when the demand changes, so does the set
of allocations that meet it. Related economic models involving evolving constraints and adjustment
mechanisms may also be found in \cite{Daher2006}. In time-varying optimization, the
minimizers of a cost that depends on time move with it
\cite{Simonetto2020}, and nonautonomous systems with time-dependent
attractors have also received considerable attention \cite{Longo2025}.
When the moving equilibrium is isolated, tracking estimates are
available. For contracting dynamics, for instance, the distance between
a solution and the equilibrium trajectory is bounded in terms of the
contraction rate and of the rate of change of the parameters
\cite{Davydov2025}. Strong contraction, however, forces the equilibrium
to be unique, so such estimates do not apply when the equilibria form a
continuum. In that case, the natural tracking error is the distance from
the state to the current equilibrium set, and the motion of this set has
to be measured in a way that remains meaningful for unbounded sets; two
nonparallel hyperplanes, for instance, are at infinite Hausdorff
distance from each other.

The aim of this paper is to develop a stability and tracking theory for
nonautonomous systems $\dot x=f(t,x)$ whose equilibrium sets
$E(t)=\{x:f(t,x)=0\}$ are nonisolated and move with time. The sets
$E(t)$ are only assumed to be nonempty and closed. We quantify their
motion on each ball by a localized Attouch--Wets distance, and we call a
locally integrable bound $v_{E,R}$ for its variation on the ball of
radius $R$ an equilibrium speed. This speed may depend on the radius,
which allows unbounded equilibrium sets that rotate. The tracking error
is the distance $\dist(t)$ from $x(t)$ to $E(t)$, and we estimate it in
two ways. The first relies on a composite Lyapunov function, in the
spirit of the dissipation framework introduced in \cite{saoud2025} for
autonomous systems, where convergence follows from the combination of
several decay mechanisms within one Lyapunov structure; here this
structure is combined with a comparison between the energy and a power
of the distance to $E(t)$. The second requires no Lyapunov function: we
only assume that each frozen vector field $f(t,\cdot)$ decreases the
distance to its own zero set $E(t)$ at a rate $\kappa>0$.

Our main results are as follows. First, under the energy--distance
comparison, we prove tube invariance and practical stability of the
moving equilibrium family when the dissipation inequality has an
integrable residual, and Lyapunov stability when the residual vanishes.
We also show that this zero-residual regime forces the equilibrium sets
to be nested, so that, for sets that genuinely move, the relevant
properties are practical stability and tracking. Second, under the
frozen decay condition, we prove that every solution of the perturbed
system $\dot x=f(t,x)+u(t)$ that remains in a ball of radius $R$ and in
the region where the condition holds satisfies
\[
\dist(t)\le e^{-\kappa(t-t_0)}\dist(t_0)
+\int_{t_0}^t e^{-\kappa(t-s)}\bigl(v_{E,R}(s)+\|u(s)\|\bigr)\,ds .
\]
Hence the tracking error tends to zero as soon as the averages of
$v_{E,R}+\|u\|$ over unit time intervals tend to zero, and we show that,
among conditions on the speed alone, this condition cannot be weakened.
Third, for comparison energies, we obtain tracking and confinement
estimates under a forced dissipation inequality, bounds in the spirit of
input-to-state stability (ISS) \cite{sontag1995}, and explicit
convergence rates under a nonlinear dissipation inequality. Finally,
when the equilibrium speeds are integrable, the sets $E(t)$ converge in
the Attouch--Wets sense to a limiting set $E_\infty$ at an explicit
rate, and tracking estimates transfer from $E(t)$ to $E_\infty$. If, in
addition, the velocity is bounded by a multiple of the distance to
$E(t)$, then every solution that remains in the region where these
conditions hold has finite length and converges to a point of
$E_\infty$, with a displacement bound that is attained in an example.
As in \cite{bhat10}, finite length is the mechanism behind convergence;
when the equilibrium set does not move, it yields the semistability of
the equilibria in the region where the hypotheses hold, in the sense of
\cite{bhat99,bhat03}.

We illustrate the results on three models. The first is a scalar
equation with a moving target, for which the condition on local
averages allows targets that escape to infinity. The second is a dynamic
resource-allocation model with time-varying demand, in which the input
enters only through its component normal to the equilibrium set and the
displacement bound of the finite-length theorem is attained. The third
is an allocation model with rotating weights: its equilibrium sets are
hyperplanes whose normal moves, no radius-independent speed exists, and
the localized speeds are needed.

The layout of the paper is as follows.
Section~\ref{sec:notation} introduces the notation, the class of
systems under consideration and the standing assumptions, including the
equilibrium speed. Section~\ref{sec:composite-analysis} is devoted to
the composite Lyapunov function and its dissipation estimates.
Section~\ref{sec:stability-tracking} contains the stability, tracking and
robustness results relative to the moving equilibrium family.
Section~\ref{sec:rates-limiting-geometry} gives convergence rates and
studies the limiting equilibrium set and the convergence of trajectories
to a point. The three models are treated in
Section~\ref{sec:examples-applications}, and
Section~\ref{sec:conclusion} concludes the paper with some directions
for future research.
%%%%%%%%%%%%%%%%%%%%%%%%%%%%%%%%%%%%%%%%%%%%%%%%%%%%%%%%%
%%%%%%%%%%%%%%%%%%%%%%%%%%%%%%%%%%%%%%%%%%%%%%%%%%%%%
\section{Notation and Framework}
\label{sec:notation}

In this section we fix the notation and introduce the class of
nonautonomous systems under consideration, together with the standing
assumption on the motion of their equilibrium sets.

Throughout this paper, $t_0\ge0$ is a fixed initial time and $\R^n$ is
the $n$-dimensional Euclidean space, equipped with the inner product
$\langle\cdot,\cdot\rangle$ and the induced norm $\|\cdot\|$. We denote
by $\overline{\B}(x,r)$ the closed ball centered at $x\in\R^n$ with
radius $r>0$.

For an interval $I\subset\R$ and $1\le p<\infty$, $L^p(I)$ denotes the
space of measurable functions $g:I\to\R$ such that
$\int_I|g(t)|^p\,dt<\infty$, and $L^p_{\loc}(I)$ denotes the space of
measurable functions whose restriction to every compact subinterval
$J\subset I$ belongs to $L^p(J)$. The space $L^\infty(I)$ consists of
the essentially bounded measurable functions on $I$, with
$\|g\|_{L^\infty(I)}:=\esssup_{t\in I}|g(t)|$, and
$W^{1,1}_{\loc}(I)$ denotes the space of locally absolutely continuous
functions $g:I\to\R$; their derivatives exist almost everywhere and
belong to $L^1_{\loc}(I)$. The same notation is used for
$\R^n$-valued functions, for instance
$L^1_{\loc}([t_0,\infty);\R^n)$. When $I=[t_0,\infty)$, we write
$\|g\|_{L^1}$ and $\|g\|_\infty$ for the norms of $g$ in $L^1(I)$ and
$L^\infty(I)$.

Given a nonempty closed set $S\subset\R^n$, the distance function to
$S$ is $\dist(x,S):=\inf_{y\in S}\|x-y\|$, $x\in\R^n$. For nonempty
closed sets $A,B\subset\R^n$, the Hausdorff--Pompeiu distance is
\[
\dH(A,B)
:=
\max\Big\{
\sup_{a\in A}\dist(a,B),
\;
\sup_{b\in B}\dist(b,A)
\Big\}.
\]

We consider the nonautonomous system
\begin{equation}\tag{P}\label{P}
\dot x(t)=f(t,x(t)),\qquad t\ge t_0,\qquad x(t_0)=x_0\in\R^n,
\end{equation}
where $f:[t_0,\infty)\times\R^n\to\R^n$ satisfies the following local
Carath\'eodory conditions: for each $x$, the map $t\mapsto f(t,x)$ is
measurable; for almost every $t$, the map $x\mapsto f(t,x)$ is locally
Lipschitz continuous; and, for each compact set $K\subset\R^n$, there
exist nonnegative functions $m_K,\ell_K\in L^1_{\loc}([t_0,\infty))$
such that
\[
\|f(t,x)\|\le m_K(t),\qquad
\|f(t,x)-f(t,y)\|\le\ell_K(t)\|x-y\|
\quad(x,y\in K)
\]
for almost every $t\ge t_0$, with a null set independent of
$x,y\in K$. These conditions give a unique maximal solution through
every initial condition, that is, a locally absolutely continuous
function that satisfies the equation almost everywhere on its maximal
interval of existence. They do not imply forward completeness. Unless a
continuation argument is given explicitly, our infinite-horizon
conclusions concern solutions defined on $[t_0,\infty)$. For
background, see \cite{khalil,Haddad}.

For each $t\ge t_0$, the equilibrium set of \eqref{P} is
\[
E(t):=\{x\in\R^n:f(t,x)=0\}.
\]
Throughout the paper, $E(t)$ is assumed to be nonempty and closed for
every $t\ge t_0$; closedness is required also at the exceptional times
at which the Carath\'eodory conditions impose no continuity in $x$.
Along a solution $x(\cdot)$, the \emph{tracking error} is
\begin{equation}\tag{TE}\label{eq:tracking-error}
\dist(t):=\dist(x(t),E(t)).
\end{equation}

The main structural assumption concerns the motion of the equilibrium
family. For nonempty closed sets $A,B\subset\R^n$ and $R>0$, the
localized Attouch--Wets distance on the ball of radius $R$ is
\[
d_R^{\rm AW}(A,B):=
\sup_{z\in\overline{\B}(0,R)}
\bigl|\dist(z,A)-\dist(z,B)\bigr|.
\]
In the notation of \cite{Rock}, this is the distance
$\mathrm{dl}_R(A,B)$.
We assume that, for every $R>0$, there exists a nonnegative function
$v_{E,R}\in L^1_{\loc}([t_0,\infty))$ such that
\begin{equation}\tag{HD}\label{HD}
d_R^{\rm AW}(E(t),E(s))
\le\int_s^t v_{E,R}(\tau)\,d\tau,
\qquad t_0\le s\le t.
\end{equation}
We call $v_{E,R}$ an \emph{equilibrium speed on the ball of radius $R$}.
It is only an upper bound and need not be unique.
Condition~\eqref{HD} controls the variation of the distance functions
uniformly on bounded sets, without truncating the equilibrium sets
themselves.

\begin{remark}[Localized and global set motion]\label{rem:local-global-speed}
For nonempty closed sets $A,B\subset\R^n$,
\[
\sup_{R>0}d_R^{\rm AW}(A,B)
=\sup_{z\in\R^n}|\dist(z,A)-\dist(z,B)|
=\dH(A,B),
\]
where the value $+\infty$ is allowed. Indeed, testing points of $A$ and
$B$ gives one inequality, while
$\dist(z,A)\le\dist(z,B)+\sup_{b\in B}\dist(b,A)$ and its symmetric
counterpart give the other, by approximate nearest points.
Consequently, a single function $v_E$ satisfies \eqref{HD} for every
$R>0$ if and only if
$\dH(E(t),E(s))\le\int_s^t v_E(\tau)\,d\tau$ for $t_0\le s\le t$.
The dependence of the speed on the radius is therefore what makes
\eqref{HD} a local condition; see
\cite[Theorems~4.35--4.36 and Corollary~4.38]{Rock}
and \cite{BeerDiConcilio1991,Beer1993}.
\end{remark}

\begin{lemma}[Absolute continuity of the tracking error]
\label{lem:tracking-error-ac}
Let $x:[s,T]\to\R^n$ be absolutely continuous, where
$t_0\le s<T<\infty$, and let $\|x(t)\|\le R$ on $[s,T]$.
Then the function $t\mapsto\dist(x(t),E(t))$ is absolutely continuous
on $[s,T]$ and, for $s\le a\le b\le T$,
\[
|\dist(b)-\dist(a)|
\le\int_a^b\bigl(\|\dot x(\tau)\|+v_{E,R}(\tau)\bigr)\,d\tau.
\]
\end{lemma}
\begin{proof}
Since the distance function is $1$-Lipschitz in its spatial argument,
\eqref{HD}, evaluated at $x(a)\in\overline{\B}(0,R)$, gives
\[
|\dist(b)-\dist(a)|
\le\|x(b)-x(a)\|+
d_R^{\rm AW}(E(b),E(a))
\le\int_a^b\bigl(\|\dot x(\tau)\|+v_{E,R}(\tau)\bigr)\,d\tau.
\]
The integrand belongs to $L^1([s,T])$. Summing this estimate over
finitely many disjoint intervals and using the absolute continuity of
the integral proves the absolute continuity of the tracking error.
\end{proof}

Since every solution is bounded on each compact subinterval of its
interval of existence, Lemma~\ref{lem:tracking-error-ac} shows that the
tracking error is locally absolutely continuous along every solution,
without any global bound on the trajectory.

To summarize, the local Carath\'eodory conditions, the nonemptiness
and closedness of every $E(t)$, and \eqref{HD} are standing assumptions.
They are in force throughout the paper and are not repeated in the
statements below. 
%%%%%%%%%%%%%%%%%%%%%%%%%%%%%%%%%%%%%%
\section{Composite Lyapunov Analysis and Dissipation Estimates}
\label{sec:composite-analysis}

In this section we introduce the Lyapunov structure on which the
comparison-energy results of the paper rely, namely the stability
results of Section~\ref{subsec:energy-stability} and the tracking
estimates of Section~\ref{subsec:composite-tracking}, and we combine
its two decay channels into a single dissipation estimate. The
construction is closely related to the composite Lyapunov framework
developed in \cite{saoud2025}. It is intended for systems in which the
distance to $E(t)$ does not decay along the frozen vector fields, so
that the decay becomes visible only through a combination of Lyapunov
functions. When the distance does decay,
Section~\ref{subsec:frozen-distance-tracking} obtains tracking
estimates for it directly, without any Lyapunov function.

Let $V_1,V_2:[t_0,\infty)\times\R^n\to[0,\infty)$ be two
Lyapunov-type functions such that, for every solution $x(\cdot)$ of
\eqref{P}, the maps $t\mapsto V_i(t,x(t))$, $i=1,2$, are locally
absolutely continuous on its interval of existence. The function $V_1$
measures the principal decay mechanism, while $V_2$ captures
directions that may not be controlled by the first decay channel
alone. With them we associate two nonnegative observables
$N_1,N_2:[t_0,\infty)\times\R^n\to[0,\infty)$ such that
$t\mapsto N_i(t,x(t))$ is measurable along every solution; they
quantify the decay appearing in the time derivatives of $V_1$ and
$V_2$. Their common zero set
\[
Z(t):=\{x\in\R^n:\;N_1(t,x)=0,\;N_2(t,x)=0\}
\]
is the moving critical set of the Lyapunov analysis. If $Z(t)$ is
larger than $E(t)$, the Lyapunov estimates give information relative
to $Z(t)$, and comparison assumptions such as the energy--distance
comparison of Section~\ref{subsec:energy-stability} are needed to
recover estimates relative to $E(t)$.

The interaction between the two decay channels is encoded in the
coupled \emph{Lyapunov dissipation} inequalities
\begin{equation}\tag{LD}\label{eq:lyapunov-dissipation}
\left\{
\begin{aligned}
\dot V_1(t,x(t))
&\le
-N_1(t,x(t))
+L(t)V_2(t,x(t))
+\rho_1(t),\\[1mm]
\dot V_2(t,x(t))
&\le
-N_2(t,x(t))
+\rho_2(t),
\end{aligned}
\right.
\end{equation}
which hold along solutions for a.e.\ $t\ge t_0$. Here
$L:[t_0,\infty)\to[0,\infty)$ is a locally essentially bounded coupling
coefficient and $\rho_1,\rho_2\in L^1_{\loc}([t_0,\infty))$ are
nonnegative residual terms. The term $L(t)V_2$ allows the first decay
channel to be affected by the second Lyapunov function, while the
second inequality provides an independent dissipation estimate for
$V_2$. These inequalities lead to the \emph{composite Lyapunov
function}
\begin{equation}\label{eq:composite-lyapunov}
W(t,x):=V_1(t,x)+\delta V_2(t,x),
\qquad \delta>0,
\end{equation}
whose weight $\delta$ is chosen so that $\delta V_2$ absorbs the
coupling term. The first result makes this precise.

\begin{theorem}[Strict composite decay]
\label{thm:composite-decay}
Assume that the coupled Lyapunov dissipation inequalities
\eqref{eq:lyapunov-dissipation} hold along every solution of \eqref{P}.
Let $L^*\ge0$ satisfy $L(t)\le L^*$ for a.e. $t\ge t_0$, and assume
that, for some $\mu>0$,
$N_2(t,x(t))\ge\mu V_2(t,x(t))$ for a.e. $t\ge t_0$.
Choose $\delta>L^*/\mu$. Then, along every solution of \eqref{P},
\[
\dot W(t,x(t))
\le
-N_1(t,x(t))
-\left(\delta-\frac{L^*}{\mu}\right)N_2(t,x(t))
+\rho_1(t)+\delta\rho_2(t).
\]
Consequently, defining
$\gamma:=\min\{1,\delta-L^*/\mu\}$,
$\Phi:=N_1+N_2$, and
$\rho:=\rho_1+\delta\rho_2$,
the composite Lyapunov function satisfies
\begin{equation}
\label{eq:strict-composite-decay}
\dot W(t,x(t))
\le
-\gamma\,\Phi(t,x(t))
+\rho(t).
\end{equation}
\end{theorem}

\begin{proof}
Let $x(\cdot)$ be a solution of \eqref{P}. The map
$t\mapsto W(t,x(t))$ is locally absolutely continuous. Adding the
first inequality in \eqref{eq:lyapunov-dissipation} to $\delta$ times
the second and using
\[
L(t)V_2(t,x(t))\le L^*V_2(t,x(t))\le\frac{L^*}{\mu}N_2(t,x(t)),
\]
we obtain the first estimate for a.e.\ $t\ge t_0$. Since
$\gamma\le1$ and $\gamma\le\delta-L^*/\mu$, it implies
\eqref{eq:strict-composite-decay}.
\end{proof}

\begin{remark}[Weight selection]
\label{rem:weight-selection}
A time-dependent weight can be used in the same way. If
$\delta:[t_0,\infty)\to(0,\infty)$ is locally absolutely continuous,
$N_2\ge\mu V_2$, and
$\delta(t)-[L(t)+\dot\delta(t)]_+/\mu\ge\eta$ for a.e.\ $t\ge t_0$
and some $\eta>0$, where $[r]_+:=\max\{r,0\}$, then
$W=V_1+\delta(t)V_2$ satisfies
$\dot W\le-N_1-\eta N_2+\rho_1+\delta(t)\rho_2$ along solutions.
We use constant weights, which avoid the additional term
$\dot\delta V_2$.
\end{remark}

Integrating \eqref{eq:strict-composite-decay} gives the following
estimates for the aggregate observable $\Phi$.

\begin{corollary}[Integral dissipation estimate]
\label{cor:integral-dissipation}
Under the hypotheses of Theorem~\ref{thm:composite-decay}, every
solution of \eqref{P} satisfies
\begin{equation}
\label{eq:integral-dissipation-bound}
\gamma\int_{t_0}^T \Phi(t,x(t))\,dt
\le
W(t_0,x(t_0))-W(T,x(T))
+\int_{t_0}^T \rho(t)\,dt
\end{equation}
for every $T$ in its interval of existence. In particular, if the
solution is global and $\rho\in L^1([t_0,\infty))$, then
\begin{equation}
\label{eq:integral-dissipation-finite}
\int_{t_0}^\infty \Phi(t,x(t))\,dt<\infty.
\end{equation}
\end{corollary}

\begin{proof}
Integrating \eqref{eq:strict-composite-decay} over $[t_0,T]$ gives
\eqref{eq:integral-dissipation-bound}. Since $W\ge0$, the right-hand
side of \eqref{eq:integral-dissipation-bound} is at most
$W(t_0,x(t_0))+\|\rho\|_{L^1}$ for every $T$, which gives
\eqref{eq:integral-dissipation-finite}.
\end{proof}

Under a uniform continuity assumption, the dissipation vanishes
asymptotically.

\begin{corollary}[Vanishing of observables]
\label{cor:observable-decay}
Under the hypotheses of Theorem~\ref{thm:composite-decay}, let $x$ be a
global solution of \eqref{P} and let $\rho\in L^1([t_0,\infty))$. If
the map $t\mapsto\Phi(t,x(t))$ is uniformly continuous on
$[t_0,\infty)$, then
\[
\Phi(t,x(t))\to0
\qquad\text{as }t\to\infty.
\]
\end{corollary}

\begin{proof}
This is Barbalat's lemma applied to $g(t):=\Phi(t,x(t))$, which is
nonnegative, uniformly continuous, and integrable on $[t_0,\infty)$ by
\eqref{eq:integral-dissipation-finite}. For completeness, suppose that
$g(t_k)\ge\varepsilon$ along a sequence $t_k\to\infty$, for some
$\varepsilon>0$. Uniform continuity gives $\eta>0$ such that
$g\ge\varepsilon/2$ on $[t_k,t_k+\eta]$ for every $k$. Passing to a
subsequence with $t_{k+1}\ge t_k+\eta$, these intervals do not
overlap, and
$\int_{t_0}^\infty g(t)\,dt\ge\sum_k\varepsilon\eta/2=\infty$,
a contradiction.
\end{proof}
%%%%%%%%%%%%%%%%%%%%%%%%%%%%%%%%%%%%%%%%%%%%%%%%%%%
\section{Stability, Tracking, and Robustness}
\label{sec:stability-tracking}

This section contains the stability and tracking results relative to
the moving equilibrium family $\{E(t)\}_{t\ge t_0}$.
Section~\ref{subsec:energy-stability} derives stability from a
comparison energy, such as the composite function
\eqref{eq:composite-lyapunov}, that is equivalent to a power of the
distance to $E(t)$. Section~\ref{subsec:frozen-distance-tracking}
derives tracking estimates for the distance itself from a decay
condition on the frozen vector fields, with a forcing term given by the
equilibrium speed; no Lyapunov function is needed there.
Section~\ref{subsec:composite-tracking} returns to comparison energies
and gives tracking and robustness estimates under a forced dissipation
inequality. For $r>0$, we denote by
\[
\mathcal T_r:=
\bigl\{(t,x)\in[t_0,\infty)\times\R^n:
\dist(x,E(t))\le r\bigr\}
\]
the tube of radius $r$ around the moving equilibrium family.

\subsection{Stability from a comparison energy}
\label{subsec:energy-stability}

In this subsection, $W$ is the composite Lyapunov function
\eqref{eq:composite-lyapunov}. We use the following
\emph{energy--distance comparison}: for constants $0<m\le M$, $p\ge1$,
and $r>0$,
\begin{equation}\tag{EDC}\label{eq:energy-distance-comparison}
m\,\dist(x,E(t))^p
\le
W(t,x)
\le
M\,\dist(x,E(t))^p
\qquad\text{for all }(t,x)\in\mathcal T_r.
\end{equation}
This inequality converts estimates for $W$ into estimates for the
tracking error \eqref{eq:tracking-error}.

\begin{definition}[Local Lyapunov stability of a moving equilibrium family]
\label{def:lyapunov-stability-moving}
The moving equilibrium family $\{E(t)\}_{t\ge t_0}$ is said to be
\emph{locally Lyapunov stable} for \eqref{P} relative to the tube
$\mathcal T_r$ if, for every $\varepsilon\in(0,r)$, there exists
$\delta_\varepsilon>0$ such that, for every solution $x(\cdot)$ of
\eqref{P}, $\dist(t_0)<\delta_\varepsilon$ implies $\dist(t)<\varepsilon$
for all $t\ge t_0$.
\end{definition}

Definition~\ref{def:lyapunov-stability-moving} extends the classical
notion of Lyapunov stability of invariant sets to the moving
equilibrium family; see, for example, \cite{Haddad}. We begin with a
tube invariance estimate that allows an integrable residual in the
energy inequality.

\begin{proposition}[Tube invariance and practical stability]
\label{prop:practical-stability-moving}
Assume that \eqref{eq:energy-distance-comparison} holds on the tube
$\mathcal T_r$, let $\varepsilon\in(0,r)$, and let $x(\cdot)$ be a
solution of \eqref{P} such that
\begin{equation}
\label{eq:W-rho-stability}
\dot W(t,x(t))
\le
\rho(t)
\qquad\text{for a.e. }t\ge t_0,
\end{equation}
where $\rho\in L^1([t_0,\infty))$ is nonnegative. If
\begin{equation}
\label{eq:practical-stability-condition}
M\,\dist(t_0)^p+\int_{t_0}^{\infty}\rho(s)\,ds
<
m\varepsilon^p,
\end{equation}
then $\dist(t)<\varepsilon$, that is, $(t,x(t))\in\mathcal T_\varepsilon$,
for every $t$ in the interval of existence of $x$. In particular,
when $\rho\equiv0$, it suffices that
$\dist(t_0)<(m/M)^{1/p}\varepsilon$.
\end{proposition}

\begin{proof}
Since $m\le M$, condition \eqref{eq:practical-stability-condition}
gives $\dist(t_0)<\varepsilon$. Suppose that $\dist(t)\ge\varepsilon$
for some $t$. Since the tracking error is continuous by
Lemma~\ref{lem:tracking-error-ac}, there is a first time $t_1>t_0$ such
that $\dist(t_1)=\varepsilon$ and $\dist(t)<\varepsilon$ on
$[t_0,t_1)$. As $\varepsilon<r$, we have $(t,x(t))\in\mathcal T_r$ on
$[t_0,t_1]$, so \eqref{eq:energy-distance-comparison} applies there.
Integrating \eqref{eq:W-rho-stability} over $[t_0,t_1]$ and using the
upper bound in \eqref{eq:energy-distance-comparison} at $t_0$, we
obtain
\[
W(t_1,x(t_1))
\le
M\,\dist(t_0)^p+\int_{t_0}^{\infty}\rho(s)\,ds
<
m\varepsilon^p,
\]
whereas the lower bound in \eqref{eq:energy-distance-comparison} at
$t_1$ gives $W(t_1,x(t_1))\ge m\,\dist(t_1)^p=m\varepsilon^p$.
This contradiction proves the claim.
\end{proof}

\begin{theorem}[Lyapunov stability of the moving equilibrium family]
\label{thm:lyapunov-stability-moving}
Assume that \eqref{eq:energy-distance-comparison} holds on the tube
$\mathcal T_r$ and that the hypotheses of
Theorem~\ref{thm:composite-decay} hold with $\rho_1=\rho_2\equiv0$.
Then the moving equilibrium family $\{E(t)\}_{t\ge t_0}$ is locally
Lyapunov stable relative to $\mathcal T_r$, with
$\delta_\varepsilon=(m/M)^{1/p}\varepsilon$.
\end{theorem}

\begin{proof}
By Theorem~\ref{thm:composite-decay},
$\dot W(t,x(t))\le-\gamma\Phi(t,x(t))\le0$ for a.e.\ $t\ge t_0$ along
every solution. For $\varepsilon\in(0,r)$,
Proposition~\ref{prop:practical-stability-moving} with $\rho\equiv0$
shows that $\dist(t_0)<(m/M)^{1/p}\varepsilon$ implies
$\dist(t)<\varepsilon$ for all $t\ge t_0$.
\end{proof}

\begin{remark}[Persistence restriction in the zero-residual regime]
\label{rem:zero-residual-persistence}
The hypotheses of Theorem~\ref{thm:lyapunov-stability-moving} impose a
restriction on the motion of the zero set. If $x(t_0)=e\in E(t_0)$,
then \eqref{eq:energy-distance-comparison} gives $W(t_0,e)=0$, and
energy monotonicity and nonnegativity imply $W(t,x(t))=0$.
Proposition~\ref{prop:practical-stability-moving} with $\rho\equiv0$,
applied for every $\varepsilon\in(0,r)$, gives $\dist(t)=0$.
Consequently $x(t)\in E(t)$, $\dot x(t)=f(t,x(t))=0$ almost
everywhere, and $x(t)\equiv e$; this constant trajectory extends to
every finite time by the local Carath\'eodory continuation argument.
Thus $E(t_0)\subseteq E(t)$ for $t\ge t_0$. If the same hypotheses and
constants hold for every initial time $s\ge t_0$, then
$E(s)\subseteq E(t)$ whenever $s\le t$, and the stability modulus
$(m/M)^{1/p}\varepsilon$ is independent of $s$. This uniform conclusion
still requires the dissipation and energy--distance assumptions;
nesting alone is not sufficient. In particular, translating or rotating
zero sets that are not nested are not covered by
Theorem~\ref{thm:lyapunov-stability-moving}. For them,
Proposition~\ref{prop:practical-stability-moving} gives practical
stability when the residual is integrable, and the tracking estimates
below quantify the effect of the motion.
\end{remark}
%%%%%%%%%%%%%%%%%%%%%%%%%%%%%%%%%%%
\subsection{Tracking from frozen distance decay}
\label{subsec:frozen-distance-tracking}

The results of Section~\ref{subsec:energy-stability} assume a
dissipation inequality for a comparison energy; they do not show how
the motion of $E(t)$ enters the estimates. In this subsection we derive
this effect directly for the distance, without assuming
\eqref{eq:lyapunov-dissipation} or \eqref{eq:energy-distance-comparison}.
We consider the perturbed system
\begin{equation}\label{eq:perturbed-moving-system}
\dot x(t)=f(t,x(t))+u(t),\qquad
u\in L^1_{\loc}([t_0,\infty);\R^n),
\end{equation}
which, under the standing assumptions, also has a unique maximal
solution through every initial condition.

Fix $R,r>0$ and $\kappa>0$. For $x\notin E(t)$, the upper frozen
directional derivative of the distance along $f(t,\cdot)$ is
\[
D_t^+(x):=
\limsup_{h\downarrow0}
\frac{\dist(x+h f(t,x),E(t))-\dist(x,E(t))}{h}.
\]
Our \emph{frozen distance decay condition} is
\begin{equation}\tag{FD}\label{eq:frozen-distance-decay}
D_t^+(x)\le-\kappa\dist(x,E(t)),
\qquad
\|x\|\le R,\quad 0<\dist(x,E(t))<r,
\end{equation}
for all $t$ outside one null set independent of $x$. In words, each
frozen field $f(t,\cdot)$ decreases the distance to its own zero set
$E(t)$ at rate $\kappa$ in the region considered. The following lemma
gives practical ways to verify this condition.

\begin{lemma}[Verification of frozen distance decay]
\label{lem:frozen-decay-tests}
Fix a time $t$ and a point $x\notin E(t)$, and put
$d_x:=\dist(x,E(t))$.
\begin{enumerate}[label=\textup{(\roman*)}]
\item If there is a nearest point $z\in E(t)$ to $x$ such that
\[
\langle f(t,x),x-z\rangle\le-\kappa d_x^2,
\]
then $D_t^+(x)\le-\kappa d_x$.
\item If $E(t)$ is convex, its metric projection $P_t$ is single-valued,
and $D_t^+(x)\le-\kappa d_x$ is equivalent to
\[
\langle f(t,x),x-P_t x\rangle\le-\kappa d_x^2.
\]
\item Suppose that $g_t:\R^n\to\R$ is convex and continuously
differentiable, $f(t,\cdot)=-\nabla g_t$,
$E(t)=\operatorname*{argmin}g_t\ne\varnothing$, and
\[
g_t(x)-\min g_t\ge\frac{\alpha}{2}d_x^2
\]
at the points under consideration, where $\alpha>0$.
Then \eqref{eq:frozen-distance-decay} holds there with
$\kappa=\alpha/2$.
\item Fix $t_*$ such that $f(t_*,\cdot)$ is locally Lipschitz.
If, from each $x$ with $\|x\|\le R$ and
$0<\dist(x,E(t_*))<r$, the solution of the frozen equation
$y'=f(t_*,y)$ satisfies
\[
\dist(y(h),E(t_*))
\le e^{-\kappa h}\dist(x,E(t_*)),\qquad y(0)=x,
\]
for all sufficiently small $h\ge0$, then
\eqref{eq:frozen-distance-decay} holds at $t_*$.
Conversely, \eqref{eq:frozen-distance-decay} at $t_*$
implies this exponential estimate on each interval on which the
frozen trajectory remains in
$\{\|y\|\le R,\ \dist(y,E(t_*))<r\}$.
\end{enumerate}
\end{lemma}

\begin{proof}
(i) In finite dimension, every nonempty closed set has a nearest point
to $x$, since a minimizing sequence is bounded. For the point $z$ in
\textup{(i)},
$\dist(x+h f(t,x),E(t))\le\|x-z+h f(t,x)\|$, and the right derivative
of the right-hand side at $h=0$ is $\langle f(t,x),x-z\rangle/d_x$.

(ii) Write $P=P_t$, which is nonexpansive, and let $q\in\R^n$.
Comparing with the point $Px$ gives
$\dist(x+q,E(t))^2\le d_x^2+2\langle x-Px,q\rangle+\|q\|^2$. Comparing
$d_x$ with the point $P(x+q)$ and using $\|P(x+q)-Px\|\le\|q\|$ gives
$\dist(x+q,E(t))^2\ge d_x^2+2\langle x-Px,q\rangle-\|q\|^2$. Since
$d_x>0$, the distance is therefore differentiable at $x$, with gradient
$(x-Px)/d_x$, and the equivalence follows.

(iii) Convexity gives
$\langle\nabla g_t(x),x-P_t x\rangle\ge g_t(x)-g_t(P_t x)
\ge\frac{\alpha}{2}d_x^2$; apply \textup{(ii)}.

(iv) Local Lipschitz continuity makes the frozen solution continuously
differentiable, so $y(h)=x+h f(t_*,x)+o(h)$, and the $1$-Lipschitz
property of the distance transfers the assumed estimate to
$D_{t_*}^+(x)$. Conversely, along a frozen trajectory in the stated
region, the distance is absolutely continuous on compact time
intervals; at its differentiability times its derivative is at most
$D_{t_*}^+(y(h))\le-\kappa\dist(y(h),E(t_*))$ where the distance is
positive, and zero where it vanishes. Multiplication by $e^{\kappa h}$
and integration give the exponential estimate.
\end{proof}

We are now ready to state the main tracking estimate of this section.

\begin{theorem}[Tracking from frozen distance decay]
\label{thm:frozen-moving-tracking}
Let $R,r,\kappa>0$, let $v_{E,R}$ be the equilibrium speed in
\eqref{HD} on the ball of radius $R$, and assume
\eqref{eq:frozen-distance-decay} on
$\{\|x\|\le R,\ 0<\dist(x,E(t))<r\}$ with a common exceptional null set.
Let $u\in L^1_{\loc}([t_0,\infty);\R^n)$ and let $x$ solve
\eqref{eq:perturbed-moving-system} on $[t_0,T]$, where $T<\infty$,
with $\|x(t)\|\le R$ and $\dist(t)<r$ throughout this interval.
Then $\dist$ is absolutely continuous and
\begin{equation}\label{eq:distance-differential-comparison}
\dot{\dist}(t)\le-\kappa\dist(t)+v_{E,R}(t)+\|u(t)\|
\quad\text{for a.e. }t\in[t_0,T].
\end{equation}
In particular,
\begin{equation}\label{eq:frozen-moving-tracking-bound}
\dist(t)\le e^{-\kappa(t-t_0)}\dist(t_0)
+\int_{t_0}^t e^{-\kappa(t-s)}
\bigl(v_{E,R}(s)+\|u(s)\|\bigr)\,ds
\quad(t_0\le t\le T).
\end{equation}
If the solution is global and satisfies the same regional conditions
for all $t\ge t_0$, the estimates hold on $[t_0,\infty)$.
\end{theorem}

\begin{proof}
\emph{Step 1: regularity and admissible times.}
Lemma~\ref{lem:tracking-error-ac} gives absolute continuity of
$\dist$. Fix $t\in(t_0,T)$ outside the null set of
\eqref{eq:frozen-distance-decay}, where $x$ and $\dist$ are
differentiable, the equation holds, and $t$ is a Lebesgue point of
$v_{E,R}$. Almost every time has these properties.

\emph{Step 2: separation of set motion and trajectory motion.}
For small $h>0$, decompose
\[
\begin{aligned}
\dist(t+h)-\dist(t)
&=\dist(x(t+h),E(t+h))-\dist(x(t+h),E(t))\\
&\quad+\dist(x(t+h),E(t))-\dist(x(t),E(t)).
\end{aligned}
\]
Because $\|x(t+h)\|\le R$, the first difference is at most
$\int_t^{t+h}v_{E,R}(s)\,ds$ by \eqref{HD}.
Differentiability of $x$ and its differential equation give
\[
x(t+h)=x(t)+h f(t,x(t))+h u(t)+o(h).
\]
The second difference is therefore at most
\[
\dist(x(t)+h f(t,x(t)),E(t))-\dist(x(t),E(t))
+h\|u(t)\|+o(h).
\]

\emph{Step 3: the differential inequality.}
If $\dist(t)>0$, divide these bounds by $h$ and take the upper
limit as $h\downarrow0$. The Lebesgue point property and
\eqref{eq:frozen-distance-decay} yield
\eqref{eq:distance-differential-comparison}.
If $\dist(t)=0$, nonnegativity and differentiability of $\dist$
at the interior minimum give $\dot{\dist}(t)=0$, and the same
inequality holds.

\emph{Step 4: integration.}
The product $e^{\kappa t}\dist(t)$ is absolutely continuous.
Multiplying \eqref{eq:distance-differential-comparison} by
$e^{\kappa t}$ and integrating over $[t_0,t]$ proves
\eqref{eq:frozen-moving-tracking-bound}.
\end{proof}

\begin{remark}[Normal input and the common-speed case]
\label{rem:normal-input}
If $E(t)$ is convex, the argument at points with $\dist(t)>0$ can
retain the signed input component:
\[
\dot{\dist}(t)\le-\kappa\dist(t)+v_{E,R}(t)
+\left\langle\frac{x(t)-P_t x(t)}{\dist(t)},u(t)\right\rangle .
\]
This holds almost everywhere on $\{\dist>0\}$; on $\{\dist=0\}$,
$\dot{\dist}=0$ almost everywhere. Thus a nonnegative locally
integrable upper bound for the signed component may replace $\|u\|$.
For an affine hyperplane only the normal component of the input
enters.

If one speed $v_E$ satisfies \eqref{HD} for all radii and
\eqref{eq:frozen-distance-decay} holds on the whole tube with the
same $\kappa$, the norm-ball restriction may be omitted from
Theorem~\ref{thm:frozen-moving-tracking}. Indeed, each finite trajectory
segment is bounded, so the preceding proof applies with its own
radius and the same speed. This removes a restriction on the
estimate, not the separate need to justify forward completeness.
\end{remark}

\begin{corollary}[Distance powers as comparison energies]
\label{cor:distance-power-energy}
Assume the hypotheses of Theorem~\ref{thm:frozen-moving-tracking},
and set $g=v_{E,R}+\|u\|$. For every $p\ge1$,
$W(t,x):=\dist(x,E(t))^p$ is absolutely continuous along the solution
and satisfies
\begin{equation}\label{eq:distance-power-forcing}
\frac{d}{dt}W(t,x(t))
\le-\kappa W(t,x(t))+\kappa^{1-p}g(t)^p
\quad\text{a.e. on }[t_0,T].
\end{equation}
It satisfies \eqref{eq:energy-distance-comparison} with $m=M=1$.
To use an $L^1$ energy comparison with the forcing in
\eqref{eq:distance-power-forcing} when $p>1$, additionally require
$g\in L^p_{\loc}([t_0,\infty))$; for $p=1$, no additional hypothesis is
needed.
\end{corollary}

\begin{proof}
The map $s\mapsto s^p$ is Lipschitz on the bounded range of the
nonnegative absolutely continuous function $\dist$, so the composition
is absolutely continuous. For $p=1$, the result is
\eqref{eq:distance-differential-comparison}. For $p>1$, the chain rule
and Young's inequality give, almost everywhere,
\[
\begin{aligned}
\frac{d}{dt}\dist(t)^p
&\le-p\kappa\dist(t)^p+p\dist(t)^{p-1}g(t),\\
p\dist(t)^{p-1}g(t)
&\le(p-1)\kappa\dist(t)^p+\kappa^{1-p}g(t)^p.
\end{aligned}
\]
This proves \eqref{eq:distance-power-forcing}.
The extra local $p$-integrability makes $g^p$ locally integrable.
\end{proof}

Thus the forced energy inequality of
Section~\ref{subsec:composite-tracking} holds for $W=\dist^p$, with
$\lambda=\kappa$ and $c=\kappa^{1-p}$; for a general composite $W$ it
must be verified separately. For $p>1$, passing through
\eqref{eq:distance-power-forcing} may lose information, since
\eqref{eq:frozen-moving-tracking-bound} requires only local
integrability of $g$.

\begin{lemma}[Exponential convolutions and local averages]
\label{lem:local-average-convolution}
Let $\alpha>0$ and let $g\ge0$ belong to
$L^1_{\loc}([t_0,\infty))$, extended by zero before $t_0$. Set
\[
G_\alpha(t):=\int_{t_0}^t e^{-\alpha(t-s)}g(s)\,ds,\qquad
S_g(\tau):=\int_\tau^{\tau+1}g(s)\,ds\quad(\tau\in\R).
\]
Write $M_g:=\sup_{\tau\in\R}S_g(\tau)$, allowing $+\infty$.
Then
\begin{equation}\label{eq:local-average-bounds}
G_\alpha(t)\le\frac{M_g}{1-e^{-\alpha}},\qquad
\limsup_{t\to\infty}G_\alpha(t)
\le\frac{\limsup_{\tau\to\infty}S_g(\tau)}{1-e^{-\alpha}}.
\end{equation}
Moreover,
\begin{equation}\label{eq:local-average-equivalence}
G_\alpha(t)\to0
\quad\Longleftrightarrow\quad
S_g(\tau)\to0.
\end{equation}
\end{lemma}

\begin{proof}
Partition the past into the intervals $[t-k-1,t-k]$, $k\ge0$,
using the zero extension of $g$. Nonnegativity gives
\[
G_\alpha(t)\le\sum_{k=0}^\infty e^{-\alpha k}S_g(t-k-1).
\]
This proves the first bound. If
$\sigma=\limsup_{\tau\to\infty}S_g(\tau)<\infty$, then $M_g<\infty$:
the averages are bounded on every compact interval by local
integrability, vanish for $\tau\le t_0-1$, and are eventually bounded.
For each integer $N\ge0$ the last display gives
\[
\limsup_{t\to\infty}G_\alpha(t)
\le\sigma\sum_{k=0}^N e^{-\alpha k}
+\frac{M_g e^{-\alpha(N+1)}}{1-e^{-\alpha}}.
\]
Let $N\to\infty$. If $\sigma=\infty$, the asserted upper bound
is immediate. Finally, for $\tau\ge t_0$,
\[
S_g(\tau)\le e^\alpha G_\alpha(\tau+1),
\]
since the kernel is at least $e^{-\alpha}$ on $[\tau,\tau+1]$.
This bound and \eqref{eq:local-average-bounds} prove
\eqref{eq:local-average-equivalence}.
\end{proof}

For any global trajectory satisfying the regional hypotheses of
Theorem~\ref{thm:frozen-moving-tracking}, with
$g=v_{E,R}+\|u\|$, this gives
\[
\limsup_{t\to\infty}\dist(t)
\le\frac{\limsup_{\tau\to\infty}S_g(\tau)}{1-e^{-\kappa}}.
\]
In particular, $S_g(\tau)\to0$ implies asymptotic tracking.
A separate confinement condition is needed if the trajectory is not
already known to stay in the region.

\begin{corollary}[Tube confinement with a supplied ball bound]
\label{cor:frozen-tube-confinement}
Let $R,r,\kappa>0$ and assume \eqref{eq:frozen-distance-decay} on the
ball of radius $R$ and the open tube of radius $r$, with a common
exceptional null set. Let $u\in L^1_{\loc}([t_0,\infty);\R^n)$, set
$g=v_{E,R}+\|u\|$, and define $M_g$ as in
Lemma~\ref{lem:local-average-convolution}.
Let $x:[t_0,T_{\max})\to\R^n$ be the maximal solution of
\eqref{eq:perturbed-moving-system}, and assume
$\|x(t)\|\le R$ on its entire maximal interval.
Suppose that at least one of the following conditions holds:
\begin{equation}\label{eq:frozen-confinement-smallness}
\begin{aligned}
&g\in L^\infty([t_0,\infty)),\qquad
\eta_\infty:=\max\{\dist(t_0),\|g\|_\infty/\kappa\}<r;\\
&\text{or}\qquad
\eta_{\rm av}:=\dist(t_0)+\frac{M_g}{1-e^{-\kappa}}<r.
\end{aligned}
\end{equation}
Then $T_{\max}=\infty$,
\eqref{eq:frozen-moving-tracking-bound} holds globally, and
$\dist(t)\le\eta_\infty<r$ or $\dist(t)\le\eta_{\rm av}<r$,
respectively. If $S_g(\tau)\to0$, then $\dist(t)\to0$.
The first condition is implied by
$\max\{\dist(t_0),(\|v_{E,R}\|_\infty+\|u\|_\infty)/\kappa\}<r$.
\end{corollary}

\begin{proof}
The tracking error is continuous by
Lemma~\ref{lem:tracking-error-ac}. Before a first tube exit
$\tau<T_{\max}$, Theorem~\ref{thm:frozen-moving-tracking} gives
\eqref{eq:frozen-moving-tracking-bound}. In the bounded-forcing case,
\[
\dist(t)\le e^{-\kappa(t-t_0)}\dist(t_0)
+\frac{\|g\|_\infty}{\kappa}(1-e^{-\kappa(t-t_0)})
\le\eta_\infty.
\]
In the average case,
Lemma~\ref{lem:local-average-convolution} gives
$\dist(t)\le\dist(t_0)+M_g/(1-e^{-\kappa})=\eta_{\rm av}$.
Letting $t\uparrow\tau$ contradicts $\dist(\tau)=r$ in either case.
Thus the relevant bound holds on the entire maximal interval.
If $T_{\max}<\infty$, then
$\|\dot x(t)\|\le m_{\overline{\B}(0,R)}(t)+\|u(t)\|$
has an integrable right-hand side up to $T_{\max}$.
Hence $x(t)$ has a finite endpoint limit, and local existence and
uniqueness extend it, a contradiction.
The last conclusion follows from
\eqref{eq:local-average-equivalence}.
\end{proof}

The ball bound is a genuine hypothesis: tube confinement alone does not
control motion tangent to an unbounded equilibrium set. The next
proposition shows that, among conditions on the speed alone, the
local-average condition is sharp.

\begin{proposition}[Sharpness of a speed-only tracking guarantee]
\label{prop:speed-only-sharpness}
Fix $\kappa>0$, a dimension $n\ge1$, and
$v\ge0$ in $L^1_{\loc}([t_0,\infty))$.
Consider all systems satisfying the local Carath\'eodory conditions
with nonempty closed zero sets $E(t)$ such that
\[
\dH(E(t),E(s))\le\int_s^t v(\tau)\,d\tau\quad(s\le t),
\]
and such that \eqref{eq:frozen-distance-decay} holds for all
$x\notin E(t)$ with rate $\kappa$, outside one common time null set.
The following statements are equivalent:
\begin{enumerate}[label=\textup{(\roman*)}]
\item $\int_t^{t+1}v(s)\,ds\to0$.
\item For every system in this class, every global solution with
$u=0$ satisfies $\dist(x(t),E(t))\to0$.
\end{enumerate}
\end{proposition}

\begin{proof}
For \textup{(i)}$\Rightarrow$\textup{(ii)}, apply
Theorem~\ref{thm:frozen-moving-tracking} on each finite interval.
The common speed and globally valid frozen decay allow a ball and
tube containing that finite segment, with constants independent of
the segment. Lemma~\ref{lem:local-average-convolution} then gives
tracking.

For the converse, put $a(t)=\int_{t_0}^t v(s)\,ds$ and consider
\[
\dot z=-\kappa(z-a(t)),\qquad \dot y=0,\qquad z(t_0)=a(t_0),
\]
where $y\in\R^{n-1}$ is omitted if $n=1$.
The vector field satisfies the Carath\'eodory conditions and has
$E(t)=\{a(t)\}\times\R^{n-1}$, exact Hausdorff speed $v$, and global
frozen decay rate $\kappa$. Its solutions are global, and
\[
\dist(x(t),E(t))
=\int_{t_0}^t e^{-\kappa(t-s)}v(s)\,ds=G_\kappa(t).
\]
Consequently
$G_\kappa(t+1)\ge e^{-\kappa}\int_t^{t+1}v$.
If \textup{(i)} fails, this solution does not track, contradicting
\textup{(ii)}.
\end{proof}

\begin{remark}[Scope of sharpness and a strict weakening]
\label{rem:average-sharpness-scope}
Proposition~\ref{prop:speed-only-sharpness} concerns the whole class
with a radius-independent speed; it does not assert necessity for an
individual system. For example, for $t\ge1$ let $a(t)=\sin(t^2)/t$
and $\dot x=-(x-a(t))$. Since $a(t)\to0$, the variation-of-constants
formula gives $x(t)\to0$, hence $|x(t)-a(t)|\to0$, although
\[
\int_t^{t+1}|\dot a(s)|\,ds\longrightarrow\frac4\pi.
\]
Indeed, $\dot a(s)=2\cos(s^2)-\sin(s^2)/s^2$, so that
$\bigl||\dot a(s)|-2|\cos(s^2)|\bigr|\le s^{-2}$; since $|\cos\xi|$
has mean value $2/\pi$ over a period, the substitution $\xi=s^2$
followed by an integration by parts gives
\[
\int_t^{t+1}\bigl(|\cos(s^2)|-2/\pi\bigr)\,ds=O(t^{-1}).
\]

The local-average condition also strictly weakens $p$-integrability.
For $\varepsilon>0$, the bounded function $v(t)=\varepsilon/\log(e+t)$
has local averages, and local averages of $v^p$, tending to zero for
every finite $p\ge1$, while $v\notin L^p([t_0,\infty))$, since
$\int_T^{2T}v(s)^p\,ds\ge\varepsilon^p T/\log(e+2T)^p\to\infty$.
In the common-speed global setting, this profile gives tracking for the
system used in the proof of
Proposition~\ref{prop:speed-only-sharpness}. Choosing $\varepsilon$
small also meets any prescribed bounded-forcing smallness condition,
although it does not remove a ball-confinement requirement in localized
results.
\end{remark}
%%%%%%%%%%%%%%%%%%%%%%%%%%%%%%%%%%%%%%%%%%%%%%%%%%%%%%%%%%%%
\subsection{Composite-energy tracking and robustness}
\label{subsec:composite-tracking}

We now return to comparison energies, including the composite function
\eqref{eq:composite-lyapunov}. Unlike
Section~\ref{subsec:frozen-distance-tracking}, the results of this
subsection assume a dissipation inequality with forcing for the energy.
Throughout the subsection, we fix $R,r>0$, $p\ge1$, $0<m\le M$ and
$\lambda>0$, and we assume the following \emph{forced energy
condition}: $W:[t_0,\infty)\times\R^n\to[0,\infty)$ satisfies
\[
m\,\dist(x,E(t))^p\le W(t,x)\le M\,\dist(x,E(t))^p
\qquad(\|x\|\le R,\ \dist(x,E(t))\le r),
\]
and, along the solution $x$ of \eqref{eq:perturbed-moving-system}
under consideration, the map $t\mapsto W(t,x(t))$ is locally absolutely
continuous and
\begin{equation}\tag{FE}\label{eq:forced-energy}
\frac{d}{dt}W(t,x(t))\le-\lambda W(t,x(t))+H(t)
\end{equation}
for almost every $t$ with $\|x(t)\|\le R$ and $\dist(t)<r$, where
$H\ge0$ is locally integrable on $[t_0,\infty)$. The forcing $H$
collects the effect of the motion of $E(t)$, of residual terms and of
inputs. Typical choices are $H=c(v_{E,R}+\rho)^p$ and
$H=c(v_{E,R}+\rho)^p+c_u\|u\|^p$, with $c,c_u\ge0$ and nonnegative
$v_{E,R},\rho\in L^p_{\loc}([t_0,\infty))$,
$u\in L^p_{\loc}([t_0,\infty);\R^n)$. By
Corollary~\ref{cor:distance-power-energy}, \eqref{eq:forced-energy}
holds for $W=\dist^p$ with $m=M=1$, $\lambda=\kappa$ and
$H=\kappa^{1-p}(v_{E,R}+\|u\|)^p$, provided that
$v_{E,R}+\|u\|\in L^p_{\loc}([t_0,\infty))$; for a general composite
energy, it is a hypothesis.

\begin{lemma}[Forced energy comparison estimate]
\label{lem:comparison-aggregate-energy}
Let $\lambda>0$ and let $H\ge0$ be locally integrable on $[t_0,\infty)$.
Let $w:[t_0,T]\to[0,\infty)$ be absolutely continuous, with $T<\infty$,
and suppose that
\begin{equation}\label{eq:W-linear-forced}
\dot w(t)\le-\lambda w(t)+H(t)
\quad\text{for a.e. }t\in[t_0,T].
\end{equation}
Then
\begin{equation}\label{eq:W-gronwall}
w(t)\le e^{-\lambda(t-t_0)}w(t_0)
+\int_{t_0}^t e^{-\lambda(t-s)}H(s)\,ds
\quad(t_0\le t\le T).
\end{equation}
The same conclusion holds on $[t_0,\infty)$ if $w$ is locally
absolutely continuous and the inequality holds almost everywhere there.
\end{lemma}

\begin{proof}
Multiplication by $e^{\lambda t}$ gives
$(e^{\lambda t}w(t))'\le e^{\lambda t}H(t)$ almost everywhere. The
product is absolutely continuous and the right-hand side is integrable
on $[t_0,T]$, so integration gives \eqref{eq:W-gronwall}. For an
infinite interval, apply this argument on every finite subinterval
$[t_0,T]$.
\end{proof}

We are now ready to state the main result of this subsection. Recall
the local averages $S_H$ and their supremum $M_H$ from
Lemma~\ref{lem:local-average-convolution}, with $H$ extended by zero
before $t_0$.

\begin{theorem}[Tracking and confinement for comparison energies]
\label{thm:quantitative-tracking}
Assume \eqref{eq:forced-energy}.
\begin{enumerate}[label=\textup{(\alph*)}]
\item \emph{Regional estimate.} If $x$ solves
\eqref{eq:perturbed-moving-system} on $[t_0,T]$, $T<\infty$, with
$\|x(t)\|\le R$ and $\dist(t)<r$ there, then, for $t\in[t_0,T]$,
\begin{equation}\label{eq:tracking-estimate-main}
\dist(t)^p\le
\frac{M}{m}e^{-\lambda(t-t_0)}\dist(t_0)^p
+\frac1m\int_{t_0}^t e^{-\lambda(t-s)}H(s)\,ds.
\end{equation}
If, moreover, $H\in L^\infty([t_0,\infty))$, the right-hand side of
\eqref{eq:tracking-estimate-main} is at most
\begin{equation}\label{eq:composite-confinement-bound}
\frac{M}{m}e^{-\lambda(t-t_0)}\dist(t_0)^p
+\frac{\|H\|_\infty}{m\lambda}\bigl(1-e^{-\lambda(t-t_0)}\bigr).
\end{equation}
If $x$ is global and these conditions hold for all $t\ge t_0$, the
estimates hold on $[t_0,\infty)$.
\item \emph{Confinement.} Let $x:[t_0,T_{\max})\to\R^n$ be the maximal
solution of \eqref{eq:perturbed-moving-system}, with $\|x(t)\|\le R$
on $[t_0,T_{\max})$, and assume that either
\begin{equation}\label{eq:composite-bounded-confinement}
H\in L^\infty([t_0,\infty))
\quad\text{and}\quad
\max\Bigl\{\frac{M}{m}\dist(t_0)^p,\ \frac{\|H\|_\infty}{m\lambda}\Bigr\}<r^p,
\end{equation}
or
\begin{equation}\label{eq:composite-average-confinement}
\frac{M}{m}\dist(t_0)^p
+\frac{M_H}{m(1-e^{-\lambda})}<r^p.
\end{equation}
Then $T_{\max}=\infty$, $\dist(t)<r$ for all $t\ge t_0$, and the
estimates of \textup{(a)} hold on $[t_0,\infty)$.
\item \emph{Asymptotic behavior.} If $x$ is a global solution with
$\|x(t)\|\le R$ and $\dist(t)<r$ for all $t\ge t_0$, in particular
under \textup{(b)}, then
\[
\dist(t)^p\le
\frac{M}{m}e^{-\lambda(t-t_0)}\dist(t_0)^p
+\frac{M_H}{m(1-e^{-\lambda})}
\qquad(t\ge t_0),
\]
\begin{equation}\label{eq:average-tracking-gain}
\limsup_{t\to\infty}\dist(t)^p
\le\frac{\limsup_{\tau\to\infty}S_H(\tau)}{m(1-e^{-\lambda})},
\end{equation}
and $\dist(t)\to0$ whenever $S_H(\tau)\to0$, in particular whenever
$H\in L^1([t_0,\infty))$.
\end{enumerate}
\end{theorem}

\begin{proof}
(a) Lemma~\ref{lem:comparison-aggregate-energy}, applied to
$w(t)=W(t,x(t))$, gives
\[
W(t,x(t))\le e^{-\lambda(t-t_0)}W(t_0,x(t_0))
+\int_{t_0}^t e^{-\lambda(t-s)}H(s)\,ds.
\]
The upper comparison bound at $t_0$ and the lower one at $t$ give
\eqref{eq:tracking-estimate-main}. If $H$ is essentially bounded, then
$\int_{t_0}^t e^{-\lambda(t-s)}H(s)\,ds\le
\|H\|_\infty(1-e^{-\lambda(t-t_0)})/\lambda$. The infinite-horizon
statement follows by applying this argument on every finite interval.

(b) The tracking error is continuous by
Lemma~\ref{lem:tracking-error-ac}, and $\dist(t_0)<r$ under either
condition, since $m\le M$. Suppose that $\dist(\tau)=r$ for some
$\tau<T_{\max}$, and let $\tau$ be the first such time. Part~(a)
applies on $[t_0,t]$ for every $t<\tau$. Under
\eqref{eq:composite-bounded-confinement}, the bound
\eqref{eq:composite-confinement-bound} is a convex combination of
$(M/m)\dist(t_0)^p$ and $\|H\|_\infty/(m\lambda)$, hence at most their
maximum. Under \eqref{eq:composite-average-confinement},
Lemma~\ref{lem:local-average-convolution} bounds the right-hand side of
\eqref{eq:tracking-estimate-main} by
$(M/m)\dist(t_0)^p+M_H/[m(1-e^{-\lambda})]$. In both cases
$\dist(t)^p$ stays below a constant smaller than $r^p$ for $t<\tau$,
which contradicts $\dist(\tau)=r$ by continuity. Hence $\dist(t)<r$ on
$[t_0,T_{\max})$. If $T_{\max}<\infty$, then
$\|\dot x(t)\|\le m_{\overline{\B}(0,R)}(t)+\|u(t)\|$ is integrable up
to $T_{\max}$, so $x(t)$ has a limit as $t\uparrow T_{\max}$ and the
solution can be continued, a contradiction. Thus $T_{\max}=\infty$.

(c) Apply Lemma~\ref{lem:local-average-convolution} to $H$ in
\eqref{eq:tracking-estimate-main}, which holds on $[t_0,\infty)$ by
part~(a); the exponential term tends to zero. If
$H\in L^1([t_0,\infty))$, its unit-interval integrals tend to zero.
\end{proof}

Part~(a) is a regional statement: it uses no boundedness or smallness
of the forcing, which enters only through the confinement tests of
part~(b). The next two remarks specialize the forcing.

\begin{remark}[Forcing generated by the set motion]
\label{rem:motion-forcing}
For $H=c(v_{E,R}+\rho)^p$, with $c\ge0$ and nonnegative
$v_{E,R},\rho\in L^p_{\loc}([t_0,\infty))$, the estimate
\eqref{eq:tracking-estimate-main} and the inequality
$(a+b)^{1/p}\le a^{1/p}+b^{1/p}$ for $a,b\ge0$ give
\begin{equation}\label{eq:tracking-estimate-split}
\begin{aligned}
\dist(t)
&\le\left(\frac{M}{m}\right)^{1/p}
e^{-\lambda(t-t_0)/p}\dist(t_0)\\
&\quad+\left(\frac{c}{m}\right)^{1/p}
\left(\int_{t_0}^t e^{-\lambda(t-s)}
\bigl(v_{E,R}(s)+\rho(s)\bigr)^p\,ds\right)^{1/p},
\end{aligned}
\end{equation}
which for $p=1$ reads
\begin{equation}\label{eq:tracking-estimate-p1}
\dist(t)\le\frac{M}{m}e^{-\lambda(t-t_0)}\dist(t_0)
+\frac{c}{m}\int_{t_0}^t e^{-\lambda(t-s)}
\bigl(v_{E,R}(s)+\rho(s)\bigr)\,ds.
\end{equation}
If $v_{E,R}$ and $\rho$ are essentially bounded and
$K:=\|v_{E,R}\|_\infty+\|\rho\|_\infty$, then $\|H\|_\infty\le cK^p$,
and \eqref{eq:composite-confinement-bound} gives
\[
\dist(t)\le
\left(\frac{M}{m}\right)^{1/p}e^{-\lambda(t-t_0)/p}\dist(t_0)
+\left(\frac{c}{m\lambda}\right)^{1/p}K,
\]
so that $\limsup_{t\to\infty}\dist(t)\le(c/(m\lambda))^{1/p}K$. If
$v_{E,R},\rho\in L^p([t_0,\infty))$, then $H\in L^1([t_0,\infty))$,
since $(a+b)^p\le2^{p-1}(a^p+b^p)$, and
Theorem~\ref{thm:quantitative-tracking}\textup{(c)} gives
$\dist(t)\to0$. Global $p$-integrability is not required: vanishing
local averages of $H$ suffice.
\end{remark}

\begin{remark}[ISS-type robustness]
\label{rem:ISS-robustness}
Input-to-state stability quantifies robustness with respect to
external inputs; see \cite{sontag1989,sontag1995}. Here the reference
object is the moving equilibrium family. Assume \eqref{eq:forced-energy}
with
\[
H=c\bigl(v_{E,R}+\rho\bigr)^p+c_u\|u\|^p,\qquad c,c_u\ge0,
\]
where $u\in L^p_{\loc}([t_0,\infty);\R^n)$ is the input in
\eqref{eq:perturbed-moving-system}. Along a global solution with
$\|x(t)\|\le R$ and $\dist(t)<r$ for all $t\ge t_0$,
Theorem~\ref{thm:quantitative-tracking}\textup{(a)} gives
\begin{equation}\label{eq:ISS-energy-estimate}
\begin{aligned}
\dist(t)^p
&\le\frac{M}{m}e^{-\lambda(t-t_0)}\dist(t_0)^p\\
&\quad+\frac1m\int_{t_0}^t e^{-\lambda(t-s)}
\left[c\bigl(v_{E,R}(s)+\rho(s)\bigr)^p+c_u\|u(s)\|^p\right]ds.
\end{aligned}
\end{equation}
Taking the power $1/p$ and using
\[
\left(\int_{t_0}^t e^{-\lambda(t-s)}\|u(s)\|^p\,ds\right)^{1/p}
\le\lambda^{-1/p}\|u\|_{L^\infty(t_0,t)},
\]
we obtain, whenever $u$ is essentially bounded on $[t_0,t]$,
\begin{equation}\label{eq:ISS-distance-estimate}
\begin{aligned}
\dist(t)
&\le C e^{-\lambda(t-t_0)/p}\dist(t_0)\\
&\quad+C\left(\int_{t_0}^t e^{-\lambda(t-s)}
\bigl(v_{E,R}(s)+\rho(s)\bigr)^p\,ds\right)^{1/p}
+C\|u\|_{L^\infty(t_0,t)},
\end{aligned}
\end{equation}
with, for instance,
\[
C=1+\left(\frac{M}{m}\right)^{1/p}
+\left(\frac{c}{m}\right)^{1/p}
+\left(\frac{c}{m\lambda}\right)^{1/p}
+\left(\frac{c_u}{m\lambda}\right)^{1/p},
\]
which depends only on $m,M,c,c_u,\lambda,p$. If $v_{E,R},\rho,u$ are
essentially bounded on $[t_0,\infty)$, then
\eqref{eq:composite-confinement-bound} gives, with the same $C$,
\begin{equation}\label{eq:ISS-limsup}
\limsup_{t\to\infty}\dist(t)
\le C\bigl(\|v_{E,R}\|_\infty+\|\rho\|_\infty+\|u\|_\infty\bigr).
\end{equation}
Theorem~\ref{thm:quantitative-tracking}\textup{(c)} gives the
local-average version of this bound, and $\dist(t)\to0$ when
$v_{E,R},\rho\in L^p([t_0,\infty))$ and
$u\in L^p([t_0,\infty);\R^n)$. Finally,
Theorem~\ref{thm:quantitative-tracking}\textup{(b)} replaces the
assumption of global existence and confinement: with
$K:=\|v_{E,R}\|_\infty+\|\rho\|_\infty$ and $U:=\|u\|_\infty$, it
suffices that
\begin{equation}\label{eq:ISS-confinement-smallness}
\max\left\{
\frac{M}{m}\dist(t_0)^p,\
\frac{cK^p+c_u U^p}{m\lambda}
\right\}<r^p,
\end{equation}
or, without essential boundedness of the forcing, that
\eqref{eq:composite-average-confinement} holds.
\end{remark}
%%%%%%%%%%%%%%%%%%%%%%%%%%%%%%%%%%%%%%%%%%%%%%%%%
\section{Convergence Rates and Limiting Equilibrium Geometry}
\label{sec:rates-limiting-geometry}

This section is devoted to convergence rates for the tracking error and
to the long-time geometry of the equilibrium family.
Section~\ref{subsec:comparison-energy-rates} obtains rates from a
comparison energy that satisfies a nonlinear dissipation inequality.
Section~\ref{subsec:limiting-equilibrium-geometry} shows that
integrable equilibrium motion produces a limiting equilibrium set, and
that the frozen distance decay condition, together with a bound of the
velocity by the distance, gives convergence of trajectories to a point
of that set; these results do not require a comparison energy.

\subsection{Rates from a comparison energy}
\label{subsec:comparison-energy-rates}

In this subsection, $W$ is nonnegative and locally absolutely continuous
along each global solution considered, and
\eqref{eq:energy-distance-comparison} holds on a tube $\mathcal T_r$.
We consider the nonlinear dissipation inequality
\begin{equation}
\label{eq:nonlinear-decay-model}
\dot W(t,x(t))
\le
-a\,W(t,x(t))^\theta+\rho(t)
\qquad\text{for a.e. }t\ge t_0,
\end{equation}
where $a>0$, $\theta>0$, and $\rho\ge0$ is a locally integrable
perturbation term. In the unforced case, $\theta=1$ gives exponential
rates and $\theta>1$ gives polynomial rates; the case $0<\theta<1$ is
discussed in Remark~\ref{rem:finite-time-theta}.

\begin{theorem}[General convergence rates]
\label{thm:general-rate}
Assume that \eqref{eq:energy-distance-comparison} holds on
$\mathcal T_r$, and let $x(\cdot)$ be a solution of \eqref{P} such that
$(t,x(t))\in\mathcal T_r$ for all $t\ge t_0$ and
\begin{equation}
\label{eq:rate-theorem-assumption}
\dot W(t,x(t))
\le
-a\,W(t,x(t))^\theta
\qquad\text{for a.e. }t\ge t_0,
\end{equation}
where $a>0$ and $\theta\ge1$. Set $w_0:=W(t_0,x(t_0))$.
\begin{enumerate}
\item[\rm(i)]
If $\theta=1$, then, for $t\ge t_0$,
\[
W(t,x(t))\le w_0\,e^{-a(t-t_0)}
\qquad\text{and}\qquad
\dist(t)
\le
\left(\frac{M}{m}\right)^{1/p}
e^{-a(t-t_0)/p}
\dist(t_0).
\]
\item[\rm(ii)]
If $\theta>1$, then
\[
W(t,x(t))\le
w_0\bigl[1+a(\theta-1)w_0^{\theta-1}(t-t_0)\bigr]^{-1/(\theta-1)},
\qquad t\ge t_0,
\]
which also holds when $w_0=0$. Consequently, there exists a constant
$C=C(a,\theta,m,w_0)>0$ such that
\[
\dist(t)
\le
C(1+t-t_0)^{-1/[p(\theta-1)]},
\qquad t\ge t_0.
\]
\end{enumerate}
In both cases, the estimate for the tracking error follows from
\eqref{eq:energy-distance-comparison}, used with its upper bound at
$t_0$ and its lower bound at $t$.
\end{theorem}

\begin{proof}
(i) If $\theta=1$, Gr\"onwall's inequality gives
$W(t,x(t))\le w_0e^{-a(t-t_0)}$. Hence
$m\,\dist(t)^p\le W(t,x(t))\le M\,\dist(t_0)^p e^{-a(t-t_0)}$, and
taking the power $1/p$ proves (i).

(ii) Let $\theta>1$. Since
$\dot W(t,x(t))\le-aW(t,x(t))^\theta\le0$, the map
$t\mapsto W(t,x(t))$ is nonincreasing; if $w_0=0$, it vanishes
identically and there is nothing to prove. Let $w_0>0$. On every
interval on which $W(t,x(t))>0$, the chain rule gives
\[
\frac{d}{dt}\Bigl(W(t,x(t))^{1-\theta}\Bigr)
=
(1-\theta)W(t,x(t))^{-\theta}\dot W(t,x(t))
\ge
a(\theta-1)
\]
for a.e.\ $t$, so that
$W(t,x(t))^{1-\theta}\ge w_0^{1-\theta}+a(\theta-1)(t-t_0)$ as long as
$W(t,x(t))>0$. Since $W$ is nonnegative and nonincreasing along $x$, it
remains zero once it vanishes. Consequently, for every $t\ge t_0$,
\[
W(t,x(t))
\le
\Bigl(
w_0^{1-\theta}
+a(\theta-1)(t-t_0)
\Bigr)^{-1/(\theta-1)},
\]
which is the asserted formula. Finally,
$w_0^{1-\theta}+a(\theta-1)(t-t_0)\ge c_0(1+t-t_0)$ with
$c_0:=\min\{w_0^{1-\theta},a(\theta-1)\}>0$, and the lower bound in
\eqref{eq:energy-distance-comparison} gives the last estimate with
$C=m^{-1/p}c_0^{-1/[p(\theta-1)]}$.
\end{proof}

\begin{remark}
\label{rem:finite-time-theta}
For $0<\theta<1$, the unforced scalar comparison equation
$\dot y=-a y^\theta$ reaches zero in finite time. Explicit finite-time
estimates are not developed here. This restriction does not apply to
Corollary~\ref{cor:integrable-perturbations}, whose qualitative
conclusion holds for every $\theta>0$.
\end{remark}

We now allow a perturbation in the dissipation inequality.

\begin{corollary}[Integrable perturbations]
\label{cor:integrable-perturbations}
Let $x$ be a global solution of \eqref{P} such that
$(t,x(t))\in\mathcal T_r$ for all $t\ge t_0$, and assume that
\eqref{eq:nonlinear-decay-model} holds along $x$ with $a>0$,
$\theta>0$ and a nonnegative $\rho\in L^1([t_0,\infty))$. Then
$W(t,x(t))\to0$ and $\dist(t)\to0$.
\end{corollary}

\begin{proof}
Set $h(t)=\int_t^\infty\rho(s)\,ds$. This function is locally
absolutely continuous, $h'=-\rho$ almost everywhere, and $h(t)\to0$.
Thus $W(t,x(t))+h(t)$ is nonnegative, locally absolutely continuous
and nonincreasing, since its derivative is at most
$-aW(t,x(t))^\theta$. It has a finite limit, and so does $W(t,x(t))$.
Integrating \eqref{eq:nonlinear-decay-model} over $[t_0,T]$ gives
\[
a\int_{t_0}^T W(t,x(t))^\theta\,dt
\le W(t_0,x(t_0))+\int_{t_0}^\infty\rho(t)\,dt.
\]
A positive limit of $W$ would make the left side unbounded for every
$\theta>0$. Hence the limit is zero, and the lower bound in
\eqref{eq:energy-distance-comparison} gives $\dist(t)\to0$.
\end{proof}

\begin{remark}[Decaying perturbations]
\label{rem:decaying-perturbations}
When $\theta=1$, the variation-of-constants estimate
\[
W(t,x(t))\le e^{-a(t-t_0)}W(t_0,x(t_0))
+\int_{t_0}^{t}e^{-a(t-s)}\rho(s)\,ds
\]
also gives rates for decaying perturbations, along a solution that
remains in $\mathcal T_r$. If $0\le\rho(t)\le Ce^{-\mu t}$ with
$C,\mu>0$, then $W(t,x(t))\le Ke^{-\nu t}$ and
$\dist(t)\le Ke^{-\nu t/p}$ for some $K>0$, where
$\nu=\min\{a,\mu\}$ if $a\ne\mu$, and any $\nu\in(0,a)$ can be taken
if $a=\mu$. Indeed, $\int_{t_0}^te^{-a(t-s)}e^{-\mu s}\,ds$ equals
$e^{-at}\bigl(e^{(a-\mu)t}-e^{(a-\mu)t_0}\bigr)/(a-\mu)$ if $a\ne\mu$
and $(t-t_0)e^{-at}$ if $a=\mu$. If $0\le\rho(t)\le C(1+t)^{-\beta}$
with $C,\beta>0$, then $W(t,x(t))\le K(1+t)^{-\beta}$ and
$\dist(t)\le K(1+t)^{-\beta/p}$: for $t\ge\max\{2t_0,1\}$, splitting the
integral at $t/2$ bounds it by
$\frac t2e^{-at/2}+\frac{2^\beta}{a}(1+t)^{-\beta}$, and the initial
term and the compact interval $[t_0,\max\{2t_0,1\}]$ are absorbed into
$K$. In both cases, the distance bound follows from the lower bound in
\eqref{eq:energy-distance-comparison}, after renaming $K$. For
$\theta>1$, the decay produced by the nonlinear dissipation competes
with that of the perturbation, and the resulting rate depends on the
relative size of $\beta$ and $1/(\theta-1)$; sharp estimates in this
regime are left for future work.
\end{remark}
%%%%%%%%%%%%%%%%%%%%%%%%%%%%%%%%%%%%%%%%%%%%%
\subsection{Limiting equilibrium geometry}
\label{subsec:limiting-equilibrium-geometry}

The results obtained so far describe the asymptotic behavior of
trajectories relative to the moving equilibrium family
$\{E(t)\}_{t\ge t_0}$. A natural question is whether the family itself
has a limiting geometric structure as $t\to\infty$. Such a structure is
needed if tracking relative to the time-varying sets $E(t)$ is to be
replaced by convergence toward a fixed limiting set.

Recall that $E(t)$ converges to a nonempty closed set $E_\infty$ in the
Attouch--Wets sense if
\[
d_R^{\rm AW}(E(t),E_\infty)\to0
\qquad\text{as }t\to\infty
\]
for every $R>0$, that is, if the distance functions
$\dist(\cdot,E(t))$ converge locally uniformly on $\R^n$ to
$\dist(\cdot,E_\infty)$. For closed sets in $\R^n$, this is equivalent
to Painlev\'e--Kuratowski convergence \cite[Theorems~4.35--4.36]{Rock};
see also \cite{BeerDiConcilio1991,Beer1993}. The following proposition shows
that integrable equilibrium speeds produce such a limit, with an
explicit rate.

\begin{proposition}[Existence and rate of a limiting equilibrium set]
\label{prop:existence-limiting-set}
Suppose that $v_{E,R}\in L^1([t_0,\infty))$ for every $R>0$. Then there
exists a nonempty closed set $E_\infty\subset\R^n$ such that
\begin{equation}\label{eq:limiting-set-local-tail}
d_R^{\rm AW}(E(t),E_\infty)
\le\int_t^\infty v_{E,R}(s)\,ds,
\qquad t\ge t_0,\quad R>0.
\end{equation}
In particular, $E(t)\to E_\infty$ in the Attouch--Wets sense.
\end{proposition}

\begin{proof}
First, \eqref{HD} evaluated at the origin, with radius $1$, gives
\[
\dist(0,E(t))
\le \dist(0,E(t_0))+\int_{t_0}^{\infty}v_{E,1}(s)\,ds=:M_0<\infty.
\]
Since nonempty closed sets in $\R^n$ attain their distance to the
origin, each $E(t)$ meets $\overline{\B}(0,M_0+1)$.

For each $R>0$, \eqref{HD} and the integrability of $v_{E,R}$ show that
$\{\dist(\cdot,E(t))\}_{t\ge t_0}$ is Cauchy in
$C(\overline{\B}(0,R))$ as $t\to\infty$. Completeness gives compatible
limits on overlapping balls, hence a nonnegative $1$-Lipschitz function
$h:\R^n\to\R$ such that $\dist(\cdot,E(t))\to h$ uniformly on bounded
sets. Set $E_\infty:=\{x:h(x)=0\}$.

Choose $t_k\to\infty$ and $z_k\in E(t_k)\cap\overline{\B}(0,M_0+1)$.
By compactness, a subsequence converges to some $z_\infty$. Local
uniform convergence and $\dist(z_\infty,E(t_k))\le\|z_\infty-z_k\|$
imply $h(z_\infty)=0$. Thus $E_\infty$ is nonempty, and it is closed
since $h$ is continuous.

Since $h$ is $1$-Lipschitz and vanishes on $E_\infty$,
$h(x)\le\dist(x,E_\infty)$. Conversely, fix $x$ and choose
$y_k\in E(t_k)$ with
$\dist(x,E(t_k))\le\|x-y_k\|\le\dist(x,E(t_k))+1/k$.
The sequence $(y_k)$ is bounded, and a subsequence converges to some
$y$. Uniform convergence on a ball containing all $y_k$ implies
$h(y_k)\to0$, so that $h(y)=0$. Moreover,
$\dist(x,E_\infty)\le\|x-y\|=h(x)$. Therefore
$h=\dist(\cdot,E_\infty)$.

Finally, fix $t$ and $R$, and let $s\to\infty$ in
$d_R^{\rm AW}(E(t),E(s))\le\int_t^s v_{E,R}(\tau)\,d\tau$. The uniform
convergence of the distance functions on $\overline{\B}(0,R)$ gives
\eqref{eq:limiting-set-local-tail}.
\end{proof}

The first step of the proof shows that all the sets $E(t)$ meet a
fixed ball; no separate assumption preventing their escape to infinity
is needed. Tracking estimates relative to $E(t)$ transfer to the
limiting set as follows.

\begin{proposition}[Transfer to a limiting equilibrium set and rates]
\label{prop:transfer-limiting-set}
Let $E_\infty\subset\R^n$ be nonempty and closed, let $R>0$, and let
$x:[t_0,\infty)\to\R^n$ satisfy $\|x(t)\|\le R$ for every $t\ge t_0$.
Then
\begin{equation}\label{eq:localized-transfer-rate}
\dist(x(t),E_\infty)
\le\dist(x(t),E(t))+d_R^{\rm AW}(E(t),E_\infty),
\qquad t\ge t_0.
\end{equation}
In particular, if $\dist(x(t),E(t))\le b(t)$ and
$d_R^{\rm AW}(E(t),E_\infty)\le\eta(t)$, then
$\dist(x(t),E_\infty)\le b(t)+\eta(t)$; and if $E(t)\to E_\infty$ in
the Attouch--Wets sense and $\dist(x(t),E(t))\to0$, then
$\dist(x(t),E_\infty)\to0$. If, moreover, $v_{E,S}\in L^1([t_0,\infty))$
for every $S>0$ and $E_\infty$ is the limiting set of
Proposition~\ref{prop:existence-limiting-set}, then
\begin{equation}\label{eq:limiting-set-trajectory-tail}
\dist(x(t),E_\infty)
\le\dist(x(t),E(t))+\int_t^\infty v_{E,R}(s)\,ds,
\qquad t\ge t_0.
\end{equation}
If a single integrable speed $v_E$ satisfies \eqref{HD} for every
radius, then $\dH(E(t),E_\infty)\le\int_t^\infty v_E(s)\,ds$, and
\eqref{eq:limiting-set-trajectory-tail} holds with $v_E$ in place of
$v_{E,R}$, without the bound $\|x(t)\|\le R$.
\end{proposition}

\begin{proof}
Evaluating the supremum that defines $d_R^{\rm AW}(E(t),E_\infty)$ at
the point $x(t)\in\overline{\B}(0,R)$ gives
\eqref{eq:localized-transfer-rate}, and the next two assertions follow
directly. Estimate \eqref{eq:limiting-set-trajectory-tail} follows from
\eqref{eq:localized-transfer-rate} and
\eqref{eq:limiting-set-local-tail}. In the common-speed case, taking
the supremum over $R$ in \eqref{eq:limiting-set-local-tail} and using
Remark~\ref{rem:local-global-speed} gives the Hausdorff bound, which
applies at every point of $\R^n$.
\end{proof}

Propositions~\ref{prop:existence-limiting-set}
and~\ref{prop:transfer-limiting-set} concern convergence toward the
limiting set. Convergence of a trajectory to a point requires, in
addition, control of its motion along the equilibrium sets, which we
obtain from a uniform bound of the velocity by the distance.

\begin{theorem}[Finite length and convergence to a limit of equilibria]
\label{thm:finite-length-convergence}
Let $R,r,\kappa>0$ and $L_f\ge0$, let $v_{E,R}$ be the equilibrium
speed in \eqref{HD} on the ball of radius $R$, and assume
\eqref{eq:frozen-distance-decay} on
$\{\|x\|\le R,\ 0<\dist(x,E(t))<r\}$ together with the
velocity--distance bound
\begin{equation}\tag{VD}\label{eq:velocity-distance-bound}
\|f(t,x)\|\le L_f\dist(x,E(t)),
\qquad \|x\|\le R,\quad \dist(x,E(t))<r,
\end{equation}
with a common exceptional null set of times. Let
$u\in L^1_{\loc}([t_0,\infty);\R^n)$ be such that
$g_R:=v_{E,R}+\|u\|\in L^1([t_0,\infty))$, and let $x$ be a global
solution of \eqref{eq:perturbed-moving-system} with $\|x(t)\|\le R$ and
$\dist(t)<r$ for every $t\ge t_0$. Set
\begin{equation}\label{eq:finite-length-budget}
\Lambda_R:=
\frac{L_f}{\kappa}
\left(\dist(t_0)+\int_{t_0}^\infty g_R(s)\,ds\right)
+\int_{t_0}^\infty\|u(s)\|\,ds.
\end{equation}
Then
\begin{equation}\label{eq:finite-length-estimates}
\int_{t_0}^\infty\dist(s)\,ds
\le\frac{\dist(t_0)+\int_{t_0}^\infty g_R(s)\,ds}{\kappa},
\qquad
\int_{t_0}^\infty\|\dot x(s)\|\,ds\le\Lambda_R.
\end{equation}
Consequently, $x(t)\to x_\infty$ for some $x_\infty\in\R^n$, with
$\|x_\infty-x(t_0)\|\le\Lambda_R$ and, more generally,
\begin{equation}\label{eq:endpoint-tail-bound}
\|x_\infty-x(t)\|
\le\frac{L_f}{\kappa}
\left(\dist(t)+\int_t^\infty g_R(s)\,ds\right)
+\int_t^\infty\|u(s)\|\,ds,
\qquad t\ge t_0.
\end{equation}
Moreover, every choice of nearest points $p(t)\in E(t)$ to $x(t)$
satisfies $p(t)\to x_\infty$. In particular, if $E(t)$ converges in the
Attouch--Wets sense to a nonempty closed set $E_\infty$, then
$x_\infty\in E_\infty$.
\end{theorem}

\begin{proof}
\emph{Step 1: integral of the tracking error.}
By Theorem~\ref{thm:frozen-moving-tracking},
$\dot{\dist}\le-\kappa\dist+g_R$ almost everywhere on every finite
interval. For $t_0\le a<T$, integration yields
\begin{equation}\label{eq:finite-interval-length-input}
\kappa\int_a^T\dist(s)\,ds
\le\dist(a)-\dist(T)+\int_a^T g_R(s)\,ds
\le\dist(a)+\int_a^T g_R(s)\,ds,
\end{equation}
and monotone convergence as $T\to\infty$ proves the first estimate in
\eqref{eq:finite-length-estimates}.

\emph{Step 2: finite length and endpoint.}
The equation and \eqref{eq:velocity-distance-bound} give
$\|\dot x(s)\|\le L_f\dist(s)+\|u(s)\|$ for a.e.\ $s\ge t_0$.
Integrating and using Step~1 gives the length estimate. For $T\ge t$,
absolute continuity gives
$\|x(T)-x(t)\|\le\int_t^T\|\dot x(s)\|\,ds$. Since the tail of an
integrable function tends to zero, $x(t)$ is Cauchy as $t\to\infty$
and has a limit $x_\infty\in\R^n$. The same estimates, started at
$a=t$, give \eqref{eq:endpoint-tail-bound} and the displacement bound.

\emph{Step 3: relation to the equilibrium sets.}
Since $g_R\in L^1([t_0,\infty))$, its local averages
$\int_\tau^{\tau+1}g_R(s)\,ds$ tend to zero as $\tau\to\infty$. Hence
Theorem~\ref{thm:frozen-moving-tracking} and
Lemma~\ref{lem:local-average-convolution} give $\dist(t)\to0$.
Nearest points exist by finite-dimensional compactness, as in the
proof of Lemma~\ref{lem:frozen-decay-tests}, and
\[
\|p(t)-x_\infty\|
\le\dist(t)+\|x(t)-x_\infty\|\to0;
\]
no measurability of the choice of $p(t)$ is required. If
$E(t)\to E_\infty$ in the Attouch--Wets sense, then
$\dist(x_\infty,E(t))\le\|x_\infty-x(t)\|+\dist(t)\to0$, and the
convergence of the distance functions at the fixed point $x_\infty$
gives $\dist(x_\infty,E_\infty)=0$. Since $E_\infty$ is closed,
$x_\infty\in E_\infty$.
\end{proof}

The limit $x_\infty$ is located through the sets $E(t)$ and their limit
$E_\infty$; it need not be a zero of $f(t,\cdot)$ for any finite $t$.
When the trajectory is not known in advance to remain in the ball and
the tube, the length budget itself can be used to confine it.

\begin{corollary}[Joint ball and tube confinement from a length budget]
\label{cor:joint-length-confinement}
Let $R,r,\kappa>0$ and $L_f\ge0$, assume
\eqref{eq:frozen-distance-decay} and \eqref{eq:velocity-distance-bound}
as in Theorem~\ref{thm:finite-length-convergence}, and let
$u\in L^1_{\loc}([t_0,\infty);\R^n)$ be such that
$g_R=v_{E,R}+\|u\|\in L^1([t_0,\infty))$. Let $x$ be the maximal
solution of \eqref{eq:perturbed-moving-system} with $x(t_0)=x_0$, and
define $\Lambda_R$ by \eqref{eq:finite-length-budget}, where
$\dist(t_0)=\dist(x_0,E(t_0))$. Assume that
\begin{equation}\label{eq:joint-ball-feasibility}
\|x_0\|+\Lambda_R<R
\end{equation}
and that $g=g_R$ satisfies one of the two conditions in
\eqref{eq:frozen-confinement-smallness}. Then $x$ is global,
$\|x(t)\|<R$ and $\dist(t)<r$ for all $t\ge t_0$, and all conclusions
of Theorem~\ref{thm:finite-length-convergence} hold.
\end{corollary}

\begin{proof}
Both inequalities hold strictly at $t_0$: $\|x_0\|<R$ by
\eqref{eq:joint-ball-feasibility}, and $\dist(t_0)<r$ by either
condition. Before a first time $\tau<T_{\max}$ at which
$\|x(\tau)\|=R$ or $\dist(\tau)=r$,
Theorem~\ref{thm:frozen-moving-tracking} gives
\eqref{eq:frozen-moving-tracking-bound} on every interval $[t_0,T]$
with $T<\tau$, and the proof of
Corollary~\ref{cor:frozen-tube-confinement} gives $\dist(t)\le\eta<r$
for $t<\tau$, where $\eta$ is $\eta_\infty$ or $\eta_{\rm av}$.
Moreover, \eqref{eq:finite-interval-length-input} and the bound
$\|\dot x\|\le L_f\dist+\|u\|$ give
\[
\|x(t)\|\le\|x_0\|+\int_{t_0}^t\|\dot x(s)\|\,ds
\le\|x_0\|+\Lambda_R<R,
\qquad t<\tau.
\]
By continuity of $\|x\|$ and $\dist$, these bounds persist at $\tau$,
which rules out both possible exits. Hence they hold on the whole
maximal interval. If $T_{\max}<\infty$, the bound $\|x(t)\|<R$ and the
continuation argument at the end of the proof of
Corollary~\ref{cor:frozen-tube-confinement} give a contradiction.
Therefore $T_{\max}=\infty$, and the hypotheses of
Theorem~\ref{thm:finite-length-convergence} hold on the whole
half-line.
\end{proof}

Condition~\eqref{eq:joint-ball-feasibility} is a genuine restriction,
since $\Lambda_R$ depends on $R$ through the speed $v_{E,R}$ and, in
general, through the admissible constant $L_f$.

\begin{remark}[The role of the velocity--distance bound]
\label{rem:velocity-bound-role}
If, for almost every $t$, $f(t,\cdot)$ is Lipschitz on
$\overline{\B}(0,R+r)$ with a constant $L_f$ independent of $t$, then
\eqref{eq:velocity-distance-bound} holds. Indeed, if $\|x\|\le R$ and
$\dist(x,E(t))<r$, a nearest point $p\in E(t)$ to $x$ satisfies
$\|p\|\le\|x\|+\dist(x,E(t))<R+r$ and $f(t,p)=0$, so that
$\|f(t,x)\|=\|f(t,x)-f(t,p)\|\le L_f\dist(x,E(t))$. Condition
\eqref{eq:velocity-distance-bound} is weaker than this uniform
Lipschitz condition, and neither follows from the local
Carath\'eodory conditions, whose Lipschitz coefficients are only
locally integrable.

The velocity--distance bound cannot be omitted. For $t\ge0$, let
\[
f(t,(z,y))=(-z,e^t z\cos t).
\]
This field satisfies the local Carath\'eodory conditions, has the fixed
zero set $E=\{0\}\times\R$, and satisfies
\eqref{eq:frozen-distance-decay} globally with $\kappa=1$ and zero
equilibrium speed. The global solution from $(1,0)$ is
$(e^{-t},\sin t)$; it is bounded and tracks $E$, but does not converge
to a point.

If a single speed $v_E$ satisfies \eqref{HD} for every radius and both
\eqref{eq:frozen-distance-decay} and \eqref{eq:velocity-distance-bound}
hold on all of $\R^n$ with fixed constants, then
$v_E+\|u\|\in L^1([t_0,\infty))$ also gives forward completeness,
without an a priori ball bound. Indeed, by
Remark~\ref{rem:normal-input}, Steps~1 and~2 of the proof of
Theorem~\ref{thm:finite-length-convergence} apply on every finite
subinterval of the maximal interval and bound the length of the
trajectory by the corresponding budget $\Lambda$, so that
$\|x(t)\|\le\|x_0\|+\Lambda$; the Carath\'eodory continuation argument
then applies.
\end{remark}

\begin{remark}[Displacement versus regularity of the endpoint map]
\label{rem:endpoint-not-lipschitz}
The bound $\|x_\infty-x_0\|\le\Lambda_R$ concerns a single trajectory;
it does not compare the endpoints of different trajectories. In fact,
the endpoint map need not be Lipschitz, even for a globally Lipschitz
autonomous field satisfying the hypotheses of
Theorem~\ref{thm:finite-length-convergence}. For
\[
f(z,y)=\bigl(-z,\min\{\max\{y,0\},|z|\}\bigr),
\]
one has $E=\{0\}\times\R$, \eqref{eq:frozen-distance-decay} holds
globally with $\kappa=1$, and
$\|f(z,y)\|\le\sqrt2\,\dist((z,y),E)$. The solution from $(1,0)$ is
$(e^{-t},0)$, with endpoint $(0,0)$. For $0<\varepsilon<1$, put
$T_\varepsilon=\frac12\log(1/\varepsilon)$. The solution from
$(1,\varepsilon)$ has $z(t)=e^{-t}$ and
\[
y(t)=
\begin{cases}
\varepsilon e^t,&0\le t\le T_\varepsilon,\\
2\sqrt{\varepsilon}-e^{-t},&t\ge T_\varepsilon,
\end{cases}
\]
so that its endpoint is $(0,2\sqrt{\varepsilon})$. Thus the ratio of
the endpoint separation to the initial separation is
$2/\sqrt{\varepsilon}$, which tends to infinity as $\varepsilon\to0$.
\end{remark}

\begin{remark}[Stability of each equilibrium for a fixed zero set]
\label{rem:fixed-set-point-stability}
Assume that $E(t)=E$ does not depend on $t$, that $u=0$, and that
\eqref{eq:frozen-distance-decay} and \eqref{eq:velocity-distance-bound}
hold as in Theorem~\ref{thm:finite-length-convergence}; the equilibrium
speed can then be taken to be zero. Fix $e\in E$ with $\|e\|<R$ and
$\varepsilon>0$, and choose $\delta_\varepsilon>0$ such that
\[
\delta_\varepsilon<r,\qquad
(1+L_f/\kappa)\delta_\varepsilon<R-\|e\|,\qquad
(1+L_f/\kappa)\delta_\varepsilon<\varepsilon.
\]
If $\|x_0-e\|<\delta_\varepsilon$, then $\dist(x_0,E)\le\|x_0-e\|$ and
$\|x_0\|+(L_f/\kappa)\dist(x_0,E)<R$. Hence
Corollary~\ref{cor:joint-length-confinement} applies and yields a global
solution with
\[
\|x(t)-e\|
\le\|x_0-e\|+\frac{L_f}{\kappa}\dist(x_0,E)
\le(1+L_f/\kappa)\|x_0-e\|<\varepsilon,
\qquad t\ge t_0.
\]
Thus every equilibrium $e\in E$ with $\|e\|<R$ is Lyapunov stable.
Moreover, by Theorem~\ref{thm:finite-length-convergence}, the solution
converges to some $x_\infty\in E$, and the last estimate gives
$\|x_\infty-e\|<R-\|e\|$, so that $\|x_\infty\|<R$ and $x_\infty$ is
itself Lyapunov stable. In the terminology of \cite{bhat99,bhat03},
each such equilibrium $e$ is therefore semistable. These
conclusions concern a fixed zero set; for a genuinely moving zero set
that is not nested, no Lyapunov stability is asserted; compare
Remark~\ref{rem:zero-residual-persistence}.
\end{remark}
%%%%%%%%%%%%%%%%%%%%%%%%%%%%%%%%%%%%%%%%%%%%%%%%%%%%%%
\section{Examples and Application}
\label{sec:examples-applications}

In this section we illustrate the results of
Sections~\ref{sec:stability-tracking}
and~\ref{sec:rates-limiting-geometry} on three finite-dimensional
models. In each of them we verify the frozen distance decay condition
\eqref{eq:frozen-distance-decay} and the velocity--distance bound
\eqref{eq:velocity-distance-bound}. The first two models admit a
radius-independent equilibrium speed. The rotating-weights model of
Section~\ref{subsec:rotating-weights} requires localized speeds and
shows how the ball bound needed for finite-length convergence can be
obtained a priori.

\subsection{A scalar moving equilibrium}
\label{subsec:scalar-moving-equilibrium}

Let $a\in W^{1,1}_{\loc}([t_0,\infty))$ and consider the scalar
equation $\dot x=-(x-a(t))$ on $\R$. The field satisfies the local
Carath\'eodory conditions, and its linear growth gives global
solutions. Here $E(t)=\{a(t)\}$, the function $v_E=|\dot a|$ satisfies
\eqref{HD} for every radius, and
\[
\langle-(x-a(t)),x-a(t)\rangle=-|x-a(t)|^2.
\]
Hence Lemma~\ref{lem:frozen-decay-tests} gives
\eqref{eq:frozen-distance-decay} globally with $\kappa=1$, and
Theorem~\ref{thm:frozen-moving-tracking}, in the common-speed form of
Remark~\ref{rem:normal-input}, yields
\begin{equation}\label{eq:scalar-A-tracking}
|x(t)-a(t)|\le e^{-(t-t_0)}|x(t_0)-a(t_0)|
+\int_{t_0}^t e^{-(t-s)}|\dot a(s)|\,ds.
\end{equation}

If $\dot a\in L^\infty([t_0,\infty))$, then
$\limsup_{t\to\infty}|x(t)-a(t)|\le\|\dot a\|_\infty$. More generally,
$\int_t^{t+1}|\dot a(s)|\,ds\to0$ implies tracking, by
Lemma~\ref{lem:local-average-convolution}. This need not entail a
bounded target or a point limit: for $\varepsilon>0$ and
$a(t)=a(t_0)+\int_{t_0}^t\varepsilon/\log(e+s)\,ds$, the target
escapes to $+\infty$ and its speed belongs to no space
$L^p([t_0,\infty))$ with $1\le p<\infty$, but the bound in
\eqref{eq:scalar-A-tracking} still tends to zero; see
Remark~\ref{rem:average-sharpness-scope}. If instead
$\dot a\in L^1([t_0,\infty))$, then $a(t)\to a_\infty$. Since
$|f(t,x)|=|x-a(t)|$, \eqref{eq:velocity-distance-bound} holds with
$L_f=1$, and Theorem~\ref{thm:finite-length-convergence}, in the
all-space common-speed case of Remark~\ref{rem:velocity-bound-role},
gives finite length and $x(t)\to a_\infty$.

\subsection{Dynamic resource allocation}
\label{subsec:dynamic-resource-allocation}

Let $n\ge2$, let $D\in W^{1,1}_{\loc}([t_0,\infty))$ be a demand, let
$a_1,\ldots,a_n>0$, and let $u\in L^1_{\loc}([t_0,\infty);\R^n)$ be an
input. We consider the allocation dynamics
\[
\dot x=-a\,\sigma(t,x)+u(t),\qquad
a=(a_1,\ldots,a_n)^\top,\qquad A=\one^\top a>0,
\]
where
\begin{equation}\label{eq:resource-mismatch}
\sigma(t,x)=\one^\top x-D(t),\qquad \one=(1,\ldots,1)^\top,
\end{equation}
is the mismatch between the total allocation and the demand. Related
resource-allocation formulations are discussed in
\cite{Arrow1958,nedic2009,Cherukuri2016}. No componentwise positivity
constraint is imposed on the state. The field is affine in $x$,
satisfies the local Carath\'eodory conditions and has linear growth,
so solutions are global.

\noindent\textbf{Verification and tracking.}
The zero set of the unforced field, its metric projection and the
distance to it are given by
\[
E(t)=\{x:\one^\top x=D(t)\},\qquad
x-P_t x=\frac{\sigma(t,x)}n\one,\qquad
\dist(x,E(t))=\frac{|\sigma(t,x)|}{\sqrt n}.
\]
Thus $\dH(E(t),E(s))=|D(t)-D(s)|/\sqrt n$, and $v_E=|\dot D|/\sqrt n$
is a common speed. Moreover,
\[
\langle-a\sigma,x-P_t x\rangle
=-\frac{A}{n}\sigma^2=-A\dist(x,E(t))^2,
\]
so that Lemma~\ref{lem:frozen-decay-tests} gives
\eqref{eq:frozen-distance-decay} globally with $\kappa=A$. By the
normal-input refinement of Remark~\ref{rem:normal-input}, the input
enters only through $|\one^\top u|/\sqrt n$, rather than $\|u\|$, and
Theorem~\ref{thm:frozen-moving-tracking} gives
\begin{equation}\label{eq:allocation-A-tracking}
\dist(t)\le e^{-A(t-t_0)}\dist(t_0)
+\frac1{\sqrt n}\int_{t_0}^t e^{-A(t-s)}
\bigl(|\dot D(s)|+|\one^\top u(s)|\bigr)\,ds.
\end{equation}
Bounded normal forcing therefore gives the corresponding $L^\infty$
gain, while vanishing local averages of
$(|\dot D|+|\one^\top u|)/\sqrt n$ give asymptotic tracking. These
statements control only the normal component of the input, whereas
Theorem~\ref{thm:finite-length-convergence} requires the
integrability of the full input.

Along solutions, the mismatch satisfies
$\dot\sigma=-A\sigma-\dot D+\one^\top u$. Consequently,
$W=\sigma^2/2=(n/2)\dist(x,E(t))^2$ satisfies
\eqref{eq:energy-distance-comparison} with $p=2$ and $m=M=n/2$, and
$\dot W\le-AW+q^2/(2A)$ almost everywhere, where
$q=|\dot D|+|\one^\top u|$; Lemma~\ref{lem:comparison-aggregate-energy}
applies when $q\in L^2_{\loc}([t_0,\infty))$, a condition that is not
needed for \eqref{eq:allocation-A-tracking}.

\noindent\textbf{Limit allocation and an attained displacement bound.}
Assume that $u=0$ and $\dot D\in L^1([t_0,\infty))$. Then
$D(t)\to D_\infty$. Moreover, \eqref{eq:velocity-distance-bound} holds
with $L_f=\sqrt n\,\|a\|$, since
$\|f(t,x)\|=\|a\|\,|\sigma(t,x)|=\sqrt n\,\|a\|\dist(x,E(t))$.
Theorem~\ref{thm:finite-length-convergence}, in the all-space
common-speed case of Remark~\ref{rem:velocity-bound-role}, gives a
trajectory of finite length. Since $\dot x$ is always parallel to $a$,
the trajectory lies on the line $x_0+\operatorname{span}\{a\}$, and its
limit belongs to $\{x:\one^\top x=D_\infty\}$. Hence
\begin{equation}\label{eq:allocation-endpoint}
x_\infty=x_0-\frac{\one^\top x_0-D_\infty}{A}\,a.
\end{equation}
In particular, for fixed $x_0$ and $a$, only the final demand enters
the limit. The estimates of Theorem~\ref{thm:finite-length-convergence}
give
\begin{equation}\label{eq:allocation-displacement-budget}
\|x_\infty-x_0\|
\le\int_{t_0}^\infty\|\dot x(t)\|\,dt
\le\frac{\|a\|}{A}
\left(|\sigma(t_0)|+\int_{t_0}^\infty|\dot D(t)|\,dt\right).
\end{equation}
Both inequalities are equalities if $\dot D$ has a constant sign
opposite to that of $\sigma(t_0)$, the case $\sigma(t_0)=0$ being
allowed. Suppose, for instance, that $\sigma(t_0)\le0$ and $\dot D\ge0$
almost everywhere. The variation-of-constants formula for
$\dot\sigma=-A\sigma-\dot D$ gives $\sigma(t)\le0$ for all $t\ge t_0$,
and integrating this equation, using $\sigma(t)\to0$, gives
$A\int_{t_0}^\infty|\sigma(t)|\,dt
=|\sigma(t_0)|+\int_{t_0}^\infty|\dot D(t)|\,dt$. The trajectory then
moves in one direction along its line, so that its length equals its
displacement; the other sign case is identical. The right-hand side
of \eqref{eq:allocation-displacement-budget} is the budget
\eqref{eq:finite-length-budget} computed with $L_f=\sqrt n\,\|a\|$ and
$\kappa=A$, which are the best constants in
\eqref{eq:velocity-distance-bound} and
\eqref{eq:frozen-distance-decay}. Hence the factor $L_f/\kappa$ in
\eqref{eq:finite-length-budget} cannot be replaced by a smaller one.

\noindent\textbf{Numerical data and comparison curves.}
For the numerical illustration, set $t_0=0$, $n=5$, and
\[
a=(0.30,0.40,0.20,0.50,0.30)^\top,\qquad
x_0=(1.0,1.5,0.8,2.0,1.2)^\top,\qquad A=1.7,
\]
with the two demand paths
\[
D_1(t)=10+2e^{-0.01t}\sin(0.30t)+1.5e^{-0.02t},
\qquad D_2(t)=10-3e^{-0.03t}.
\]
In Figure~\ref{fig:tracking-error}, the demand is $D_1$, and we compare
the persistent input $u_i(t)=0.08\sin(\omega_i t)$ with the vanishing
input $u_i(t)=0.08e^{-0.025t}\sin(\omega_i t)$, where
$\omega=(0.5,0.8,1.1,1.4,1.7)^\top$. Since
$|\dot D_1(t)|\le2\sqrt{0.0901}\,e^{-0.01t}+0.03e^{-0.02t}$ and
$|\one^\top u(t)|\le0.4e^{-\eta t}$, with $\eta=0$ or $\eta=0.025$,
estimate \eqref{eq:allocation-A-tracking} gives the comparison curves
\begin{equation}\label{eq:allocation-numerical-comparison}
b_\eta(t)=e^{-At}\dist(0)+\frac1{\sqrt5}
\left(2\sqrt{0.0901}J_{A,0.01}(t)
+0.03J_{A,0.02}(t)+0.4J_{A,\eta}(t)\right),
\end{equation}
where, for $\alpha>0$, $\mu\ge0$ and $t\ge0$,
\begin{equation}\label{eq:exponential-convolution-kernel}
J_{\alpha,\mu}(t)=\int_0^t e^{-\alpha(t-s)}e^{-\mu s}\,ds
=\begin{cases}
(e^{-\mu t}-e^{-\alpha t})/(\alpha-\mu),&\alpha\ne\mu,\\
t e^{-\alpha t},&\alpha=\mu.
\end{cases}
\end{equation}
For the persistent input, $b_0(t)$ tends to
$0.4/(A\sqrt5)\approx0.1052$, since the contribution of the decaying
demand vanishes in the limit; this value bounds
$\limsup_{t\to\infty}\dist(t)$. For the vanishing input, both the input
and the demand derivative are integrable, and the tracking error tends
to zero. On $[400,500]$, the maxima of the two computed tracking errors
are approximately $0.087$ and $0.0026$, respectively.

In Figure~\ref{fig:limit-allocation}, $u=0$ for both demand paths.
Since $D_1$ and $D_2$ have the same limit $D_\infty=10$,
\eqref{eq:allocation-endpoint} gives the common limit
\[
x_\infty\approx(1.6176,2.3235,1.2118,3.0294,1.8176)^\top.
\]
For $D_2$, we have $\sigma(0)=-0.5$ and $\dot D_2>0$, so both sides of
\eqref{eq:allocation-displacement-budget} are approximately equal to
$1.6341$.

\begin{figure}[!htb]
\centering
\includegraphics[width=0.9\linewidth]{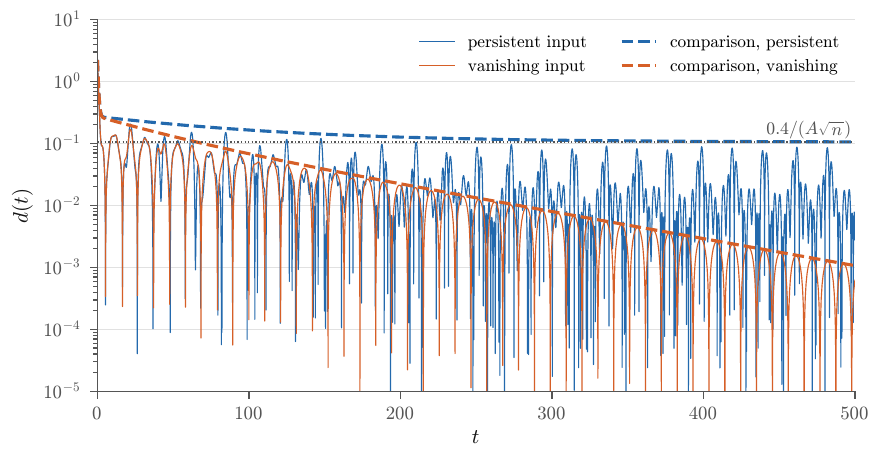}
\caption{Allocation tracking with demand $D_1$ for the persistent and
vanishing inputs, together with the comparison curves
\eqref{eq:allocation-numerical-comparison}. The dotted horizontal line
marks the limit $0.4/(A\sqrt5)$ of $b_0$. The vertical axis is
logarithmic; values below $10^{-5}$ lie outside its range.}
\label{fig:tracking-error}
\end{figure}

\begin{figure}[!htb]
\centering
\includegraphics[width=0.9\linewidth]{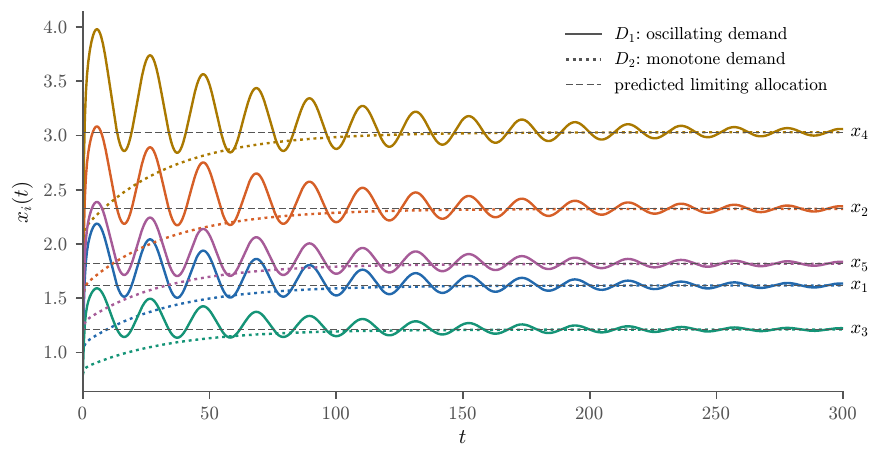}
\caption{Unforced allocations for $D_1$ (solid) and $D_2$ (dotted). The
horizontal dashed lines mark the limiting allocation
\eqref{eq:allocation-endpoint}, not the values at the final time
$t=300$. For the monotone demand $D_2$, the displacement bound is
attained.}
\label{fig:limit-allocation}
\end{figure}
%\FloatBarrier

\subsection{Rotating weights and localized equilibrium motion}
\label{subsec:rotating-weights}

Let $n\ge2$, let $w\in W^{1,1}_{\loc}([t_0,\infty);\R^n)$ with
$\|w(t)\|=1$ for every $t$, and let $D\in W^{1,1}_{\loc}([t_0,\infty))$.
Let $a:[t_0,\infty)\to\R^n$ be measurable and essentially bounded, with
$a(t)\ne0$ for every $t$, and put
$A_0=\esssup_{t\ge t_0}\|a(t)\|$. Assume that
$w(t)^\top a(t)\ge\kappa>0$ for almost every $t$ and that
$\dot w,\dot D\in L^1([t_0,\infty))$. We consider
\[
\dot x=-a(t)\sigma(t,x),\qquad \sigma(t,x)=w(t)^\top x-D(t).
\]
The field satisfies the local Carath\'eodory conditions and has global
solutions, since $\|f(t,x)\|\le A_0(\|x\|+|D(t)|)$ almost everywhere.
Since $a(t)\ne0$ for every $t$, the zero set at every time is
$E(t)=\{x:w(t)^\top x=D(t)\}$, and $\dist(x,E(t))=|\sigma(t,x)|$.

\noindent\textbf{Set motion and frozen decay.}
On the ball of radius $R$, the distance functions satisfy
\[
\bigl|\dist(z,E(t))-\dist(z,E(s))\bigr|
\le R\|w(t)-w(s)\|+|D(t)-D(s)|,
\qquad \|z\|\le R,
\]
so \eqref{HD} holds with $v_{E,R}=R\|\dot w\|+|\dot D|$. This speed
depends on the radius. Indeed, if $w(t)\ne\pm w(s)$, the hyperplanes
$E(t)$ and $E(s)$ are not parallel and their Hausdorff distance is
infinite, so that no radius-independent speed exists once the normal
rotates. At points where $\sigma(t,x)\ne0$, a direct computation of the
frozen derivative gives
\[
D_t^+(x)=-(w(t)^\top a(t))|\sigma(t,x)|
\le-\kappa\dist(x,E(t)).
\]
Thus \eqref{eq:frozen-distance-decay} holds on all of $\R^n$, and so
does \eqref{eq:velocity-distance-bound}, with $L_f=A_0$, since
$\|f(t,x)\|\le A_0\dist(x,E(t))$.

\noindent\textbf{An a priori ball bound.}
Along a solution,
$\dot\sigma=\dot w^\top x-(w^\top a)\sigma-\dot D$. The function
$t\mapsto|\sigma(t,x(t))|$ is locally absolutely continuous, and
\[
\frac{d}{dt}|\sigma|
\le-\kappa|\sigma|+\|\dot w\|\|x\|+|\dot D|
\quad\text{a.e.}
\]
Integrating this inequality and using
$\|x(t)\|\le\|x_0\|+A_0\int_{t_0}^t|\sigma(s)|\,ds$, we obtain
\[
\|x(t)\|\le
\|x_0\|+\frac{A_0}{\kappa}
\left(|\sigma(t_0)|+\|\dot D\|_{L^1}\right)
+\frac{A_0}{\kappa}\int_{t_0}^t\|\dot w(s)\|\|x(s)\|\,ds.
\]
Gr\"onwall's inequality therefore yields
\begin{equation}\label{eq:rotating-priori-radius}
\|x(t)\|\le
\left(\|x_0\|+\frac{A_0}{\kappa}
\bigl(|\sigma(t_0)|+\|\dot D\|_{L^1}\bigr)\right)
\exp\left(\frac{A_0}{\kappa}\|\dot w\|_{L^1}\right)
=:\mathcal R .
\end{equation}

\noindent\textbf{Tracking and a point limit.}
Fix $R\ge\mathcal R$, so that $\|x(t)\|\le R$ for all $t\ge t_0$ by
\eqref{eq:rotating-priori-radius}. Since \eqref{eq:frozen-distance-decay} and
\eqref{eq:velocity-distance-bound} hold on all of $\R^n$, the tube
restriction in Theorems~\ref{thm:frozen-moving-tracking}
and~\ref{thm:finite-length-convergence} is void: any
$r>\sup_{t\ge t_0}\dist(t)$ can be used, and this supremum is finite
because $\dist(t)\le\|x(t)\|+|D(t)|$ and $D$ is bounded.
Theorem~\ref{thm:frozen-moving-tracking} gives
\[
\dist(t)\le e^{-\kappa(t-t_0)}\dist(t_0)
+\int_{t_0}^t e^{-\kappa(t-s)}
\bigl(R\|\dot w(s)\|+|\dot D(s)|\bigr)\,ds.
\]
Since the speed $v_{E,R}$ on this fixed ball is integrable,
Theorem~\ref{thm:finite-length-convergence} gives $x(t)\to x_\infty$
and
\[
\|x_\infty-x_0\|\le
\frac{A_0}{\kappa}
\left(\dist(t_0)+R\|\dot w\|_{L^1}
+\|\dot D\|_{L^1}\right).
\]
The limits $w_\infty$ and $D_\infty$ exist, $\|w_\infty\|=1$, and the
same distance-function estimate on each ball shows that $E(t)$
converges in the Attouch--Wets sense to
$E_\infty=\{x:w_\infty^\top x=D_\infty\}$. Thus $x_\infty\in E_\infty$.
If $a(t)=a$ is constant, the trajectory lies on the line
$x_0+\operatorname{span}\{a\}$, and
\begin{equation}\label{eq:rotating-endpoint}
x_\infty=x_0-
\frac{w_\infty^\top x_0-D_\infty}{w_\infty^\top a}\,a,
\qquad w_\infty^\top a\ge\kappa.
\end{equation}

\noindent\textbf{Planar illustration.}
Let $t_0=0$, $a=(1,0.5)^\top$, $x_0=(3,-2)^\top$,
\[
w(t)=(\cos\theta(t),\sin\theta(t))^\top,\quad
\theta(t)=1-e^{-t},\qquad D(t)=1+0.5e^{-0.2t}\sin(2t).
\]
The function $\theta\mapsto\cos\theta+\tfrac12\sin\theta$ is concave on
$[0,1]$, so its minimum over this interval is attained at an endpoint,
and $\kappa=\cos1+\tfrac12\sin1\approx0.961038$. Here $A_0=\sqrt5/2$ and
$\|\dot w\|_{L^1}=1$. Writing $\phi=\arctan10$, the extrema of $D$ occur
at $\tau_k=(\phi+k\pi)/2$, $k\ge0$, where
$D(\tau_k)-1=\tfrac12(-1)^k\sin\phi\,e^{-\phi/10}e^{-k\pi/10}$. Since
$D(0)=1$ and $D$ is monotone on $[0,\tau_0]$ and between consecutive
extrema,
\[
\|\dot D\|_{L^1}
=|D(\tau_0)-1|+\sum_{k\ge0}|D(\tau_{k+1})-D(\tau_k)|
=\frac{\sin\phi\,e^{-\phi/10}}{1-e^{-\pi/10}}
\approx3.185910.
\]
Thus \eqref{eq:rotating-priori-radius} gives
$\mathcal R\approx30.8501<31$. For Figure~\ref{fig:rotating-weights}
we take $R=31$ and use $|\dot D(t)|\le\sqrt{1.01}\,e^{-0.2t}$, which
gives the comparison curve
\begin{equation}\label{eq:rotating-numerical-comparison}
b_R(t)=2e^{-\kappa t}+31J_{\kappa,1}(t)
+\sqrt{1.01}\,J_{\kappa,0.2}(t).
\end{equation}
Finally, \eqref{eq:rotating-endpoint} gives
$x_\infty\approx(4.10509,-1.44745)^\top$. The script that generates
Figures~\ref{fig:tracking-error}--\ref{fig:rotating-weights} is
provided as supplementary material.

\begin{figure}[!htb]
\centering
\includegraphics[width=0.98\linewidth]{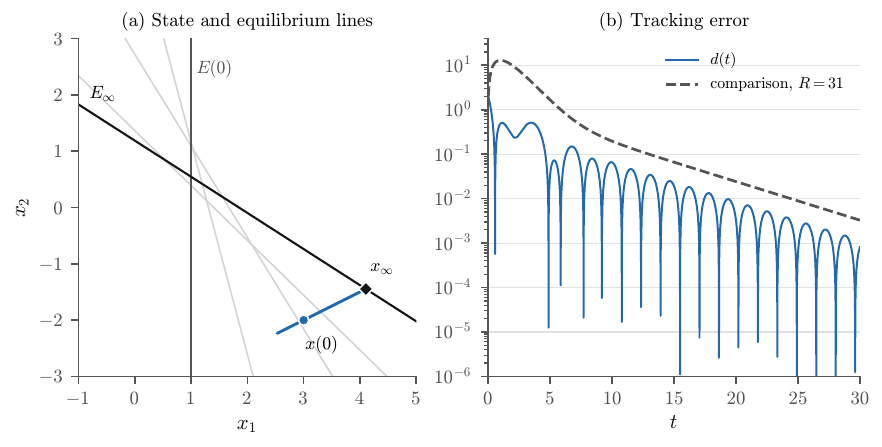}
\caption{Rotating weights. (a) The trajectory on $[0,30]$, selected
equilibrium lines $E(t)$, and the limit point $x_\infty$ of
\eqref{eq:rotating-endpoint}, marked by a diamond. (b) The tracking
error and the comparison curve \eqref{eq:rotating-numerical-comparison}
with $R=31$. The vertical axis is logarithmic; values below $10^{-6}$
lie outside its range.}
\label{fig:rotating-weights}
\end{figure}
\FloatBarrier
%%%%%%%%%%%%%%%%%%%%%%%%%%%%%%%%%%%%%%%
\section{Conclusion}
\label{sec:conclusion}

The results of this paper show that, for equilibrium sets that move
without being nested, the relevant properties are practical stability
and tracking rather than Lyapunov stability, and that the motion of the
sets enters the tracking estimates through the localized equilibrium
speed. This speed controls the tracking error through an exponentially
weighted average, and its averages over unit time intervals govern
asymptotic tracking. When it is integrable, the equilibrium sets
converge to a limiting set, and a bound of the velocity by the distance
then gives the convergence of confined trajectories to a point. The
examples show that the speed must be allowed to depend on the radius,
as for hyperplanes with rotating normals, and that the displacement
bound of the finite-length theorem cannot be improved in general.

The present analysis is carried out for differential equations. Its
extension to differential inclusions, in particular to sweeping
processes driven by a moving set and to systems governed by maximally
monotone operators, is natural, since moving equilibrium sets arise
directly in these settings and may have a richer geometry. These
questions will be investigated in future work.
%%%%%%%%%%%%%%%%%%%%%%%
\backmatter

%\bmhead{Supplementary information}

%If your article has accompanying supplementary file/s please state so here. 

%Authors reporting data from electrophoretic gels and blots should supply the full unprocessed scans for key as part of their Supplementary information. This may be requested by the editorial team/s if it is missing.

%Please refer to Journal-level guidance for any specific requirements.

\bmhead{Acknowledgements}

%Acknowledgements are not compulsory. Where included they should be brief. Grant or contribution numbers may be acknowledged.

%Please refer to Journal-level guidance for any specific requirements.

\section*{Declarations}

\bmhead{Conflict of interest}
The authors declare no competing interests.

\bmhead{Ethics approval} 
Not applicable.

\bmhead{Availability of data and materials}
Not applicable.

\bibliography{sn-bibliography}% common bib file
%% if required, the content of .bbl file can be included here once bbl is generated
%%\input sn-article.bbl
\end{document}